\documentclass[1 [leqno,11pt]{amsart}
\usepackage{amssymb, amsmath, color,mathrsfs,enumerate}
\usepackage{hyperref}

\let\al=\alpha

\let\f=\frac
\let\p=\psi

\let\na=\nabla

\let\pa=\partial

\newcommand{\beq}{\begin{equation}}
	\newcommand{\eeq}{\end{equation}}
\newcommand{\ben}{\begin{eqnarray}}
	\newcommand{\een}{\end{eqnarray}}
\newcommand{\beno}{\begin{eqnarray*}}
	\newcommand{\eeno}{\end{eqnarray*}}

\newtheorem{thm}{Theorem}[section]
\newtheorem{lem}{Lemma}[section]
\newtheorem{rmk}{Remark}[section]
\newtheorem{cor}{Corollary}[section]
\newtheorem{prop}{Proposition}[section]

\numberwithin{equation}{section}

\newcommand{\andf}{~\text{ and }~}
\newcommand{\with}{~\text{ with }~}

\def\eqdef{\buildrel\hbox{\footnotesize def}\over =}
\def\avg{\operatorname{avg}}

\def\dive{{\mathop{\rm div}\nolimits}\,}

\let\f=\frac
\let\p=\partial

\def\N{\mathop{\mathbb  N\kern 0pt}\nolimits}
\def\Q{\mathop{\mathbb  Q\kern 0pt}\nolimits}
\def\R{{\mathop{\mathbb R\kern 0pt}\nolimits}}
\def\Z{\mathop{\mathbb  Z\kern 0pt}\nolimits}
\def\PP{\mathop{\mathbb P\kern 0pt}\nolimits}
\renewcommand{\epsilon}{\varepsilon}
\def\ds{\displaystyle}
\def\pa{\partial}
\def\hat{\widehat}

\allowdisplaybreaks

\begin{document}
\title[The inhomogenenous Navier-Stokes equations]
{On the density patch problem for the 2-D inhomogeneous Navier-Stokes equations}

\author[T. Hao]{Tiantian Hao}
\address[T. Hao]{
School of Mathematical Sciences, Peking University, Beijing 100871, China}
\email{haotiantian@pku.edu.cn}

\author[F. Shao]{Feng Shao}
\address[F. Shao]{School of Mathematical Sciences, Peking University, Beijing 100871,  China}
\email{fshao@stu.pku.edu.cn}

\author[D. Wei]{Dongyi Wei}
\address[D. Wei]{
School of Mathematical Sciences, Peking University, Beijing 100871, China}
\email{jnwdyi@pku.edu.cn}

\author[Z. Zhang]{Zhifei Zhang}
\address[Z. Zhang]{
School of Mathematical Sciences, Peking University, Beijing 100871, China}
\email{zfzhang@math.pku.edu.cn}

\date{\today}

\begin{abstract}
In this paper, we first construct a class of global strong solutions for the 2-D inhomogeneous Navier-Stokes equations under very general assumption that the initial density is only bounded and the initial velocity is in $H^1(\R^2)$. With suitable assumptions on the initial density, which includes the case of density patch and vacuum bubbles, we prove that Lions' s weak solution is the same as the strong solution with the same initial data. In particular,  this gives a complete resolution of the density patch problem proposed by Lions \cite{PL}: {\it for the density patch data $\rho_0=1_{D}$ with a smooth bounded domain $D\subset\R^2$, the regularity of $D$ is preserved by the time evolution of Lions's weak solution.}
\end{abstract}

\maketitle



\section{Introduction}

We consider the inhomogeneous incompressible Navier-Stokes equations in $ \mathbb{R}^{+}\times\Omega, \Omega\subseteq \mathbb{R}^{d}, d=2, 3$:
\begin{equation}\label{INS}
     \left\{
     \begin{array}{l}
     \partial_t\rho+u\cdot\nabla \rho=0,\\
     \rho(\partial_tu+u\cdot\nabla u)-\Delta u+\nabla P=0,\\
     \dive u = 0,\\
     (\rho, u)|_{t=0} =(\rho_{0}, u_0),
     \end{array}
     \right.
\end{equation}
where $\rho,~u$ stand for the density and velocity of the fluid respectively, and $P$ is a scalar pressure function. 

Lady\v{z}enskaja and Solonnikov \cite{LS} first addressed the question of unique solvability of \eqref{INS} in a bounded domain $\Omega$. Under the assumption that $u_0\in W^{2-\frac{2}{p},p}(\Omega)(p>d)$ is divergence free and vanishes on $\p\Omega$ and  $\rho_0\in C^1(\Omega)$ is bounded away from zero, they proved the global well-posedness in dimension $d=2$, and local well-posedness in dimension $d=3$. We refer to \cite{A, AG, AGZ, AP, D, DM2, HPZ, PZ} and references therein for the well-posedness results in the critical functional framework.  In all these works, the density has to be continuous, bounded and bounded away from zero.

In the case when $\rho_0\in L^{\infty}(\R^d)$ with a positive lower bound and $u_0\in H^1(\R^d)$, Kazhikov \cite{K} proved the global existence of weak solution to the system \eqref{INS}, and Simon \cite{S} removed the lower bound assumption on $\rho_0$. While with $u_0\in H^2(\R^d)$, Danchin and Mucha \cite{DM3} proved that the system \eqref{INS} has a unique local in time solution. Paicu, Zhang and the fourth author \cite{PZZ}  proved the global well-posedness of \eqref{INS} with $\rho_0\in L^{\infty}(\R^2)$ bounded away from zero and  $u_0\in H^s(\R^2)$ for any $s>0$, and $\|u_0\|_{L^2}\|\na u_0\|_{L^2}$ small  in dimension 3. Furthermore, Zhang \cite{Z} established the global existence of strong solution with $\rho_0\in L^{\infty}(\R^3)$ bounded away from zero and $u_0$ small in the critical Besov space $\dot{B}^{\f12}_{2,1}(\R^3)$. See \cite{DW, D-apde} for recent important progress in this direction.

\subsection{Lions's density patch problem} 

In this paper,  we only focus on the case of $\Omega=\R^2$, since the global regularity is still a famous open problem even for the 3-D incompressible Navier-Stokes equations.

Given $0\leq \rho_0\in L^{\infty}(\R^2)$, and $u_0$ satisfying $\dive u_0=0,\,\sqrt{\rho_0}u_0\in L^2(\R^2)$, Lions \cite{PL} proved that the system \eqref{INS} has a global weak solution satisfying the energy inequality 
\begin{align}\label{energy}
\f12\|\sqrt{\rho}u(t)\|_{L^2}^2+\int_0^t\|\nabla u(s)\|^2_{L^2}\,ds\leq \f12\|\sqrt{\rho_0}u_0\|_{L^2}^2.
\end{align}
Moreover, for any $0\leq \alpha\leq \beta<\infty$, the Lebesgue measure $\mu (\{x\in \R^2:\alpha\leq \rho(t,x)\leq \beta\})$ is independent of $t$.
Furthermore, Lions proposed the following open question(see page 34 in \cite{PL}): 

{\it We would like to mention another interesting  open question: suppose that $\rho_0=1_{D}$ for a smooth domain $D(\subset \Omega)$, i.e., a patch of a homogeneous incompressible fluid {\color{red}``}surrounded" by the vacuum(or a bubble of vacuum embedded in the fluid). Then, Theorem 2.1 yields at least one global weak solution and (2.17) implies that, for all $t\ge 0, \rho(t)=1_{D(t)}$ for some set such that $vol(D(t))=vol(D)$. In this case, (2.1)-(2.2) can be reformulated as a somewhat complicated free boundary problem. It is also very natural to ask whether the regularity of $D$ is preserved by the time evolution.}\smallskip

The density patch type problem was first solved by Liao and Zhang \cite{LZ-ARMA, LZ} for the initial density of the form
$$
\rho_0=\eta_11_{\Omega_0}+\eta_21_{\Omega_0^c},
$$
where  $(\eta_1,\eta_2)$ is any pair of positive constants and $\Omega_0$ is a bounded, simply connected domain with $W^{3,p}$ boundary, see \cite{DZ, GG} for related works. 
When $\eta_2=0$, the first breakthrough is due to Danchin and Mucha \cite{DM}, where they constructed a class of strong solutions for  $\rho_0\in L^\infty(\Omega)$ and $u_0\in H^1(\Omega)$ without extra compatibility condition in the case when the fluid domain $\Omega$ is either a bounded domain or the Torus; in particular, for the density patch data, the regularity of $D\subset \Omega$ is preserved by the time evolution of strong solution they constructed. Prange and Tan \cite{PT} extend some results in \cite{DM} to the case of $\Omega=\R^2$; moreover, for the density with a bubble of vacuum(i.e., $\eta_1=0, \eta_2>0$), they can prove the weak-strong uniqueness between Lions's weak solution and strong solution they constructed, thus the regularity of $D$ is preserved by the time evolution of  Lions's weak solution in such case. However, for the density patch type data(i.e., $\eta_1>0, \eta_2=0$),  \cite{PT} still need to impose the compatibility condition \eqref{comp}  on the initial data and a weighted $L^\infty$ bound for the weak solution. Thus, the following original version of Lions' s density patch problem remains open:\smallskip

{\it For the density patch data $\rho_0=1_{D}$ with a smooth bounded domain $D\subset\R^2$, whether the regularity of $D$ is preserved by the time evolution of Lions's weak solution?}
\smallskip

This problem is analogue to the vortex patch problem, which was solved by Chemin \cite{Che}, see also \cite{BC} for a new proof.

\subsection{Main results}

The goal of this paper is twofolds. First of all, we prove the global existence of strong solution under very general assumption on the initial density. 

\begin{thm}\label{thm:strong}
Given the initial data $(\rho_0,u_0)$ satisfying $0\leq \rho_0(x)\leq C_0 \text{~and~}\rho_0\not\equiv 0,~ u_0\in H^{1}(\R^2),
~\dive u_0=0$, the system \eqref{INS} has a global solution $(\rho,u,\nabla P)$ 
satisfying the weak energy inequality \eqref{energy}, $0\leq \rho(t,x)\leq C_0$,
and the following properties:

\begin{itemize}
    \item  $\sqrt{\rho}u,~\nabla u\in L^{\infty}(\R^{+};L^2(\R^2))$;
    \item $ \sqrt{\rho}\dot{u},~\nabla^2u,~ \nabla P,~ \sqrt{t}\nabla \dot{u}
    \in L^2(\R^{+};L^2(\R^2))$, here $\dot{u}\eqdef u_t+u\cdot \nabla u$;
    \item 
    $\nabla u\in L^{1}([0,T];L^{\infty}(\R^2)),$ $\rho\in \mathcal{C}([0,T];L^p(B_R))$ for all $1\leq p<\infty,~R,~ T\in(0,\infty)$.
\end{itemize}

\end{thm}

\begin{rmk}

\begin{itemize}

\item[(1)]Compared with the existence part  in \cite{PT}, here we only require $0\leq \rho_0(x)\leq C_0$ and $\rho_0\not\equiv 0$\footnote{It's natural to assume $\rho_0\not\equiv 0$, otherwise the problem is trivial and the solution is $\rho\equiv 0$, $u\equiv 0$.}. For the case of density patch data(FFV-2 condition in Theorem A, \cite{PT}), except for $u_0\in H^1(\R^2)$\footnote{Indeed, they assume the initial velocity satisfies $\sqrt{\rho_0}u_0\in L^2(\R^2)$ and $\nabla u_0\in L^2(\R^2)$.}, \cite{PT} also imposed the following key compatibility condition
\begin{align}\label{comp}
-\Delta u_0+\nabla P_0=\sqrt{\rho_0}g\quad \text{for}\,\, u_0\in L^2(\R^2),\,\nabla u_0\in L^1(\R^2),\,\,g\in L^2(\R^2).
\end{align}

\item[(2)]
In \cite{CK, HW, LSZ}, the initial density allows the vacuum but requires $\rho_0\in W^{1,q}, q>2$, which excludes the data of density patch or vacuum bubbles.

\item[(3)] In this paper, many new ideas are introduced and the method is robust. We believe that our method could be applied to the well-posed problem for the compressible fluids in the presence of vacuum, see \cite{HLX, LZ-cns} for example.

\item[(4)]
As $\na u \in L^1(0,T;L^\infty(\R^2))$ for any $T>0$,  there exists a unique continuous flow $X$ associated with $u$ globally defined on $\R^+\times \R^+\times \R^2$ by
\begin{align}\label{trajectory}
X(t,s,x)=x+\int_s^t u(\sigma, X(\sigma,s,x))\,d\sigma.    
\end{align}
Then $\rho(t,x)=\rho_0(X(0,t,x))$.

\end{itemize}

\end{rmk}

To solve the density patch problem, we need to explore the further regularity of strong solution under one of the following two assumptions:
\begin{itemize}

\item[(H1)] $\rho_0$ has a compact support, thus $(1+|x|^2)\rho_0\in L^1(\R^2)$.

\item[(H2)] there exist $R_0,~c_0\in (0,\infty)$ such that $\int_{B(x_0,R_0)}\rho_0\,dx\geq c_0>0$ for all $x_0\in \R^2$. 

\end{itemize}

\begin{rmk}
Indeed,  (H1) corresponds to the density patch problem, and (H2) allows vacuum bubbles, especially countable equal size vacuum bubbles. While, the  following condition in \cite{PT}
\begin{align*}
\max\{\underline{\rho}-\rho_0,0\}\in L^p(\R^d), \text{~for some~} \underline{\rho}\in(0,\infty), ~p\in(d/2,\infty),    
\end{align*}
only allows finitely many vacuum bubbles with an equal non-zero measure.
\end{rmk}

\begin{thm}\label{nabla^2 u_Lm}
Under the assumptions of Theorem \ref{thm:strong}, if $\rho_0$ satisfies (H1) or (H2), then for all $2\leq m<\infty,~T>0$, 
the solution satisfies $\sqrt{t}\nabla^2u\in L^2([0,T];L^m(\R^2))$.    
\end{thm}

As a corollary, we can prove that the regularity of $D$ is preserved by the time evolution of strong solution with the initial density $\rho_0=a+b\mathbf{1}_{D}$.

\begin{cor}\label{cor1}
Under the assumptions of Theorem \ref{thm:strong}, if $\rho_0=a+b\mathbf{1}_{D}$ for some bounded domain $D\subset \R^2$ with $C^{1,\alpha}$ regularity $(0<\alpha<1)$ 
and $a\geq 0,~a+b\geq 0,~(a,b)\in \R^2\setminus \{(0,0)\}$, then the solution provided by Theorem \ref{thm:strong} is such that for all $t\geq 0,~\rho(t,\cdot)=a+b\mathbf{1}_{D_t}$,
with $D_t\eqdef X(t,0,D)$, where $X(t,s,\cdot)$ stands for the flow of $u$ that is the unique solution of \eqref{trajectory}. Furthermore, $D_t$ has $C^{1,\alpha}$ regularity.  
\end{cor}

\begin{rmk} In fact, the open set $D$ in Corollary \ref{cor1} is allowed to be a disjoint union of finitely many open connected sets $D^{(j)}(1\leq j\leq N)$ with $\overline{D^{(i)}}\cap \overline{D^{(j)}}=\emptyset$ for $i\neq j$. In this case, there holds
    \[D_t=\bigcup_{j=1}^N D_t^{(j)}\quad\text{with}\quad \overline{D^{(i)}_t}\cap \overline{D^{(j)}_t}=\emptyset\text{ for }i\neq j\quad\forall\ t\geq 0.\]
 This follows directly from the fact that    
 \begin{equation}\label{Eq.X_low_bound}
        |X(t, 0, x_1)-X(t, 0, x_2)|\geq |x_1-x_2|\exp\Big(-\int_0^t\|\na u(s)\|_{L^\infty}ds\Big),\quad \forall\ x_1, x_2\in \R^2,\ \forall\, t>0.
    \end{equation}
Indeed, let $\varphi(t)=X(t, 0, x_1)-X(t, 0, x_2)$. By \eqref{trajectory}, we have
\beno
|\varphi'(t)|\le \|\na u(t)\|_{L^\infty}|\varphi(t)|.
\eeno
Then Gr\"onwall's inequality ensures that 
\beno
|\varphi(t)|\ge |\varphi(0)|\exp\Big(-\int_0^t\|\na u(s)\|_{L^\infty}ds\Big).
\eeno

 \end{rmk}

The second goal of this paper gives a complete resolution of  Lions's density patch problem via proving the weak-strong uniqueness and Corollary \ref{cor1}.

\begin{thm}\label{thm_weak_unique}
Under the assumptions of Theorem \ref{thm:strong}, if $\rho_0$ satisfies (H1) or (H2), then the weak solution in \cite{PL} is unique and coincides with strong solution in Theorem \ref{thm:strong}. In particular, for the density patch data or vacuum bubbles, the regularity of $D$ is preserved by the time evolution of Lions's weak solution.
\end{thm}

\begin{rmk}
 At this stage, we only prove the uniqueness of weak solutions  for $\rho_0$ satisfying (H1) or (H2). This gives a partial answer to Lions's open question about the uniqueness of weak solutions in two dimensions (see page 31 in \cite{PL}).
\end{rmk}

\subsection{Sketch of the proof}

First of all, it is relatively standard to obtain the following energy estimates 
\begin{align*}
\f12\|\sqrt{\rho}u(t)\|_{L^2}^2+\int_0^t\|\nabla u(s)\|^2_{L^2}\,ds\leq \f12\|\sqrt{\rho_0}u_0\|_{L^2}^2,
\end{align*}
and  for $i=0, 1$, 
\begin{equation}\nonumber\begin{aligned}
&\sup_{t\in [0,T]} t^i\|\nabla u\|^2_{L^2}+\int_0^Tt^i\int_{\R^2}\rho|\dot{u}|^2\,dxdt
\leq C,\\
&\int_0^Tt^i\big(\|\nabla^2 u\|_{L^2}^2+\|\nabla P\|^2_{L^2}\big)\,dt\leq C.
\end{aligned}
\end{equation}
Here and in what follows, the constant $C$ depends only on $\|\rho_0\|_{L^{\infty}},~\|\sqrt{\rho_0}u_0\|_{L^2}$ and $\|\nabla u_0\|_{L^2}$. 

Based on {\bf  an important observation $\p_iu^k\p_ku^j\p_ju^i=0$}  due to an algebraic fact $\text{Tr}(A^3)=0$ if $\text{Tr}(A)=0$ for a $2\times 2$ matrix $A$,  we can further show that for $i=1,2$(see Lemma \ref{dot u})
 \begin{equation*}
 \begin{aligned}
&\sup_{t\in [0,T]} t^i\|\sqrt{\rho} \dot{u}\|^2_{L^2}
+\int_0^Tt^i\|\nabla\dot{u}\|_{L^2}^2\,dt\leq C,\\
&\sup_{t\in [0,T]} t^i\bigl(\|\nabla^2u\|^2_{L^2}+\|\nabla P\|^2_{L^2}\bigr)\leq C.
\end{aligned}\end{equation*}

However, due to the presence of vacuum of the density, it is highly nontrivial to obtain  $u\in L^\infty_{loc}$ and $\na u\in L^1(0,T; L^\infty(\R^2))$ from the above energy estimates,
which are crucial to prove the existence of strong solution and weak-strong uniqueness.

To prove $u\in L^\infty_{loc}$, we use the following propagation property of lower bound of the density in the sense that  if $\int_{B(x_0,R_0)} \rho_0\,dx\geq c_0>0$, then for $R\geq R_0+2t\|\sqrt{\rho_0}u_0\|_{L^2}/\sqrt{c_0}$,  
\beno
\int_{B(x_0,R)} \rho(t,x)\,dx\geq c_0/4.
\eeno
Let's emphasize that this property also holds for weak solutions(see Lemma \ref{lem2}).  This lower bound ensures that there exists a $x_t\in B(x_0,R(t))$ such that $|u(t,x_t)|\leq C_1$. On the other hand, from the energy estimates, we can deduce that
\begin{equation}\nonumber
|u(t, x)-u(t, y)|\leq C(1+t)^{-\f12}\left[\ln \left(2+|x-y|/\sqrt{t}\right)\right]^{\f12}.
\end{equation}
Thus, we can conclude the following bound 
\begin{align*}
|u(t,x)|\leq C[\ln(2+|x-x_0|+t^{-1})]^{\f12} \quad \forall\ x\in\R^2,\,\, t>0.
\end{align*}
This bound is completely new and may be {of independent} interest. See Lemma \ref{u_L^infty}  for the detail. 

To prove $\na u\in L^1(0,T; L^\infty(\R^2))$, {\bf the key ingredient} is to establish the following {pointwise inequality} of $\na u(t,x)$, which is also completely new:
\begin{align*}
|\nabla u(t,x)|\leq& Cr^{-1}\|\nabla u(t)\|_{L^2}
+C\int_{B(x,r)}\frac{|\rho \dot{u}(t,y)|}{|y-x|}\,dy.
\end{align*}
where we further have 
\begin{align*}
\int_{B(x,r)}\frac{|\rho \dot{u}(t,y)|}{|y-x|}\,dy
\leq Cr\|\nabla \dot{u}\|_{L^2}+C\|\sqrt{\rho}\dot{u}\|_{L^2}.
\end{align*}
See Lemma \ref{Lip} for the detail. 

To prove the $C^{1,\al}$ regularity propagation of the density patch or vacuum bubbles, we need to improve the regularity of $u$.  Indeed,  if $\rho_0$ satisfies (H1) or (H2), then for all $2\leq m<\infty$, we have $\sqrt{t}\nabla^2u\in L^2([0,T];L^m(\R^2))$.  For this, the key point is to use {\bf the following important  inequality}(see Lemma \ref{lem4})
 \begin{equation}\nonumber
  \|\sqrt{\rho}(t, \cdot)f\|_{L^m(\R^2)}\leq C\|(\sqrt{\rho}(t, \cdot)f,\nabla f)\|_{L^2(\R^2)}\quad \forall\ t\in[0,T],
\end{equation}
such that 
\begin{align*}
        \|\sqrt{t}\nabla^2u\|_{L^2([0,T];L^m(\R^2))}&\leq C\|\sqrt{t}\rho\dot{u}\|_{L^2([0,T];L^m(\R^2))}\leq C\|\sqrt{t}\sqrt\rho\dot{u}\|_{L^2([0,T];L^m(\R^2))}\\
        &\leq C\|(\sqrt {t\rho}\dot{u}, \sqrt t\nabla \dot u)\|_{L^2([0,T];L^2(\R^2))}\leq C.
    \end{align*}
    
For the weak-strong uniqueness, we follow the argument of Prange and Tan \cite{PT}.  Let $(\bar{\rho}, \bar{u})$ be the strong solution obtained in Theorem \ref{thm:strong} and $(\rho, u)$ be the weak solution provided by Lions in \cite{PL} with same initial data $(\rho_0, u_0)$. We denote $\delta\!\rho=\rho-\bar{\rho},  \delta\!u=u-\bar{u}$.  First of all, we have
\begin{align*}
&D(t)\leq \|\rho_0\|^{\f12}_{L^{\infty}}\|\sqrt{\rho}\delta\! u\|_{L^2((0,t)\times\R^2)}\exp\Big({\|\nabla \bar{u}\|_{L^1(0,T;L^{\infty}(\R^2))}}\Big)\leq C\|\sqrt{\rho}\delta\! u\|_{L^2((0,t)\times\R^2)},\label{Dt2}
\end{align*}
where $D(t)=\ds\sup_{0<s\leq t}s^{-\f12}\|\delta\!\rho(s,\cdot)\|_{\dot{H}^{-1}(\R^2)}.$ Using a duality argument, we can show that
\begin{align*}
\|\sqrt{\rho}\delta\! u\|^2_{L^2((0,T)\times\R^2)}\leq \int_0^T\bigl|\langle\delta\!\rho\dot{\bar{u}},v\rangle\bigr|\,dt+\int_0^T\bigl|\langle\rho\delta\!u\cdot\nabla\bar{u},v\rangle\bigr|\,dt,
\end{align*}
where $v$ is the solution of the linear backward parabolic system \eqref{Eq.v_backward}. By Lemma \ref{lem4}, we can prove that 
\begin{align*}
&\int_0^T\bigl|\langle\rho\delta\!u\cdot\nabla\bar{u},v\rangle\bigr|\,dt\leq\int_0^T\|\sqrt{\rho}\delta\!u\|_{L^2} \|\sqrt{\rho}v\|_{L^4} \|\nabla\bar{u}\|_{L^{4}}\,dt\leq CT^{\f14}\|\sqrt{\rho}\delta\! u\|^2_{L^2((0,T)\times\R^2)}.
\end{align*}
Compared with \cite{PT},  {\bf the main innovation} is to control the trouble term $\int_0^T\bigl|\langle\delta\!\rho\dot{\bar{u}},v\rangle\bigr|\,dt$, which is based on the trilinear estimates in Lemma \ref{uni_H^-1_bubble} and Lemma \ref{uni_H^-1_patch},  in a way as 
\begin{align*}
\int_0^T\bigl|\langle\delta\!\rho\dot{\bar{u}},v\rangle\bigr|\,dt 
&\leq C \|(\sqrt{\rho}v,\nabla v)\|_{L^{\infty}(0,T;L^2)}
\int_0^T D(t)|\ln ({\sqrt{t}}D(t))|^{\f12}\|\bigl(t^{\f12}\sqrt{\bar{\rho}}\dot{\bar{u}},t^{\f12}\nabla\dot{\bar{u}}\bigr)\|_{L^2}\,dt.
\end{align*}
Thus, we can conclude    
\begin{align*}
D(t)\leq C \int_0^t D(s)|\ln ({\sqrt{s}}D(s))|^{\f12}\bigl(\|s^{\f12}\sqrt{\bar{\rho}}\dot{\bar{u}}\|_{L^2}+s^{\f12}\|\nabla\dot{\bar{u}}\|_{L^2}\bigr)\,ds.
\end{align*}
Then Osgood's lemma ensures that $D(t)=0$ on $[0,T]$.

\medskip

\noindent{\bf Notations.}\smallskip

\begin{enumerate}[(i)]
    \item[(1)] $\dot{u}\eqdef u_t+u\cdot \nabla u$ is the material derivative of the velocity.
    \item[(2)] $D_t\eqdef{(\p_t+u\cdot \nabla)},~\langle x \rangle\eqdef\sqrt{1+|x|^2}$.
    \item[(3)] $B(a,R)\eqdef\{x\in \R^2:|x-a|<R\},\, B_R=B(0,R), ~\forall\ a \in \R^2,~ R>0$.
    \item[(4)] For matrices $A=(A_{ij})$ and $B=(B_{ij})$, we denote $A:B=\sum_{i,j} A_{ij}B_{ij}$.
    \item[(5)] $ |D|$ is the Lebesgue measure of $D$ for $D\subset\R^2$.
    \item[(6)] We shall always denote $C$ to be an absolute constant which may vary from line to line. The dependence of the constant $C$ will be explicitly indicated if there are any exceptions.
\end{enumerate}

\section{Some important properties of weak solutions}

In this section, we prove some important properties of weak solutions $(\rho, u)$ only using the density equation and the energy inequality \eqref{energy}. 
These properties are also very useful for strong solutions.

\begin{lem}\label{lem1}
For any $ \varphi\in C_c^{\infty}(\R^2)$, we have 
\beno
\big|\|\varphi\sqrt{\rho}(t)\|_{L^2}-\|\varphi\sqrt{\rho}_0\|_{L^2}\big|\leq t\|\sqrt{\rho_0}u_0\|_{L^2}\|\nabla\varphi\|_{L^{\infty}}.
\eeno
\end{lem}

\begin{proof}
Testing the density equation of \eqref{INS} against $\varphi^2$ gives 
\[\int_{\R^2}\rho(t,x)\varphi^2(x)\,dx-\int_{\R^2}\rho_0(x)\varphi^2(x)\,dx=\int_0^t\int_{\R^2}(\rho u)(s,x)\cdot\nabla(\varphi^2)(x)\,dx\,ds\quad\forall\ t\geq 0.\]
Let $I(t)=\|\sqrt{\rho(t)}\varphi\|_{L^2}$. Then by H\"older's inequality, we have
\begin{align*}
    I(t)^2&\leq I(0)^2+2\int_0^t\|\sqrt{\rho}u(s)\|_{L^2}\|\sqrt{\rho(s)}\varphi\|_{L^2}\|\nabla\varphi\|_{L^\infty}\,ds\\
    &\leq I(0)^2+2\|\sqrt{\rho_0}u_0\|_{L^2}\|\nabla\varphi\|_{L^\infty}\int_0^tI(s)\,ds,
\end{align*}
 where we have used the energy inequality \eqref{energy}  in the last inequality. Now it follows from Gr\"onwall's lemma that 
\beno
I(t)\leq I(0)+t\|\sqrt{\rho_0}u_0\|_{L^2}\|\nabla\varphi\|_{L^{\infty}},
\eeno
that is,
\[\|\sqrt{\rho(t)}\varphi\|_{L^2}\leq \|\sqrt{\rho_0}\varphi\|_{L^2}+t\|\sqrt{\rho_0}u_0\|_{L^2}\|\nabla\varphi\|_{L^{\infty}}\quad\forall\,\,t\geq 0.\]

On the other hand, since the density equation of \eqref{INS} is time-reversible, we also have 
\[\|\sqrt{\rho_0}\varphi\|_{L^2}\leq \|\sqrt{\rho(t)}\varphi\|_{L^2}+t\|\sqrt{\rho}u(t)\|_{L^2}\|\nabla\varphi\|_{L^{\infty}}\leq \|\sqrt{\rho(t)}\varphi\|_{L^2}+t\|\sqrt{\rho_0}u_0\|_{L^2}\|\nabla\varphi\|_{L^{\infty}}\]
for all $t\geq 0$, where we used the energy inequality \eqref{energy} again. 
\end{proof}

As $\rho_0\not\equiv 0$, there exist $x_0\in \R^2$ and $R_0\in (0,\infty)$ such that
\begin{align}\label{c0}
\int_{B(x_0,R_0)} \rho_0\,dx\geq c_0>0.    
\end{align}

\begin{lem}\label{lem2}
Assume \eqref{c0} and $R\geq R_0+2t\|\sqrt{\rho_0}u_0\|_{L^2}/\sqrt{c_0}$. Then we have 
\beno
\int_{B(x_0,R)} \rho(t,x)\,dx\geq c_0/4.
\eeno
\end{lem}
\begin{proof}
Fix $t>0$ and $R\geq R_0+2t\|\sqrt{\rho_0}u_0\|_{L^2}/\sqrt{c_0}$. For any $\varepsilon>0$, we choose a bump function $\varphi\in C_c^\infty(\R^2;[0, 1])$ such that $\varphi|_{B(x_0, R_0)}=1$, $\operatorname{supp}\varphi\subset B(x_0, R)$ and $\|\nabla\varphi\|_{L^\infty}\leq (1+\varepsilon)(R-R_0)^{-1}\leq \frac{(1+\varepsilon)\sqrt{c_0}}{2t\|\sqrt{\rho_0}u_0\|_{L^2}}$.  By Lemma \ref{lem1}, we have
\begin{align*}
    \|\sqrt{\rho(t)}\varphi\|_{L^2}\geq \|\sqrt{\rho_0}\varphi\|_{L^2}-t\|\sqrt{\rho_0}u_0\|_{L^2}\|\nabla\varphi\|_{L^\infty}\geq \|\sqrt{\rho_0}\varphi\|_{L^2}-\frac{1+\varepsilon}{2}\sqrt{c_0}.
\end{align*}
It follows from the support properties of $\varphi$ that $$\|\sqrt{\rho(t)}\varphi\|_{L^2}\leq \left(\int_{B(x_0,R)} \rho(t,x)\,dx\right)^{1/2}, \quad \|\sqrt{\rho_0}\varphi\|_{L^2}\geq \left(\int_{B(x_0,R_0)} \rho_0(x)\,dx\right)^{1/2}\geq \sqrt{c_0},$$
hence
\[\left(\int_{B(x_0,R)} \rho(t,x)\,dx\right)^{1/2}\geq \sqrt{c_0}-\frac{1+\varepsilon}{2}\sqrt{c_0}.\]
Letting $\varepsilon\to0^+$ completes the proof.
\end{proof}

\begin{lem}\label{lem3}
If $\langle x \rangle\rho_0\in L^2(\R^2)$, then there holds
\beno
\|\langle x \rangle\sqrt{\rho}(t)\|_{L^2}\leq\|\langle x \rangle\sqrt{\rho}_0\|_{L^2}+ t\|\sqrt{\rho_0}u_0\|_{L^2}.
\eeno
\end{lem}

\begin{proof}
Let $\eta\in C_c^\infty(\R)$ be such that $\eta(r)=r$ for $r\in[-1,1]$ and $|\eta'|\leq 1$. Then $|\eta(r)|\leq r$ for $r\in\R$. For any $R>1$, we define $\varphi_R(x)=R\eta(\langle x\rangle/R)$ for $x\in\R^2$, then $\varphi_R\in C_c^\infty(\R^2)$, $|\varphi_R(x)|\leq\langle x\rangle $ and $\nabla\varphi_R(x)=\frac{x}{\langle x\rangle}\eta'(\langle x\rangle/R)$, thus $|\nabla\varphi_R|\leq 1$. 
    By Lemma \ref{lem1}, we have
    \begin{align*}
        \|\varphi_R\sqrt{\rho}(t)\|_{L^2}\leq\|\varphi_R\sqrt{\rho}_0\|_{L^2}+ t\|\sqrt{\rho_0}u_0\|_{L^2}\leq\|\langle x \rangle\sqrt{\rho}_0\|_{L^2}+ t\|\sqrt{\rho_0}u_0\|_{L^2}.
    \end{align*}
    Therefore, letting $R\to +\infty$ and using $\varphi_R(x)=\langle x\rangle$ for $|x|\leq R-1$ completes the proof.
\end{proof}

\begin{lem}\label{lem4}Let $p\in [2,\infty)$ and $T\in(0,\infty)$. If $\rho_0$ satisfies (H1), then there exists a constant $C>0$ depending only on $T, p, \|\rho_0\|_{L^{\infty}}, \|\sqrt{\rho_0} \langle x\rangle\|_{L^2}, \|\sqrt{\rho_0}u_0\|_{L^2}$, such that
 (for $f\in \dot{H}^1(\R^2)$)
 \begin{equation}\label{lem4.weak_estimate_H1}
 C^{-1} \|\sqrt{\rho}(t, \cdot)f\|_{L^p(\R^2)}\leq  \|f\|_{L^2(B_1)}+\|\nabla f\|_{L^2(\R^2)}\leq 
C\|(\sqrt{\rho}(t, \cdot)f,\nabla f)\|_{L^2(\R^2)},\ \forall\ t\in[0,T].
\end{equation}If $\rho_0$ satisfies (H2),
then there exists a constant $C>0$ depending only on  $T$, $p$, $R_0$, $c_0$, $\|\rho_0\|_{L^{\infty}}$, $\|\sqrt{\rho_0}u_0\|_{L^2}$ such that
 (for $f\in \dot{H}^1(\R^2)$)
 \begin{equation}\label{lem4.weak_estimate_H2}
    \|f\|_{H^1(\R^2)}+\|\sqrt{\rho}(t, \cdot)f\|_{L^p(\R^2)}\leq 
C\|(\sqrt{\rho}(t, \cdot)f,\nabla f)\|_{L^2(\R^2)},\quad \forall\ t\in[0,T].
\end{equation}
\end{lem}
\begin{proof}
We first show \eqref{lem4.weak_estimate_H1}. By Lemma \ref{lem3} and Lemma \ref{lem_fm}, we get  
\begin{align}
&\|\sqrt{\rho}f\|_{L^p(\R^2)}
=\|\sqrt{\rho}\langle x\rangle^{\frac{1}{p}}\langle x\rangle^{-\frac{1}{p}}f\|_{L^p(\R^2)}
\leq \|\sqrt{\rho}\langle x\rangle^{\frac{1}{p}}\|_{L^{\frac{3p}{2}}(\R^2)}\|\langle x\rangle^{-\frac{1}{p}}f\|_{L^{3p}(\R^2)}\nonumber\\
&=\bigg(\int_{\R^2}\rho^{\f{3p}{4}}  \langle x\rangle^{\f32}\,dx\bigg)^{\frac{2}{3p}} 
\bigg(\int_{\R^2}\frac{|f|^{3p}} {\langle x\rangle^{3}}\,dx\bigg)^{\frac{1}{3p}}\nonumber\\
&\leq C\|\rho\|^{\f12-\frac{2}{3p}}_{L^{\infty}}\bigg(\int_{\R^2}\rho  \langle x\rangle^{2}\,dx\bigg)^{\frac{2}{3p}}
\Bigl(\|f\|_{L^2(B_1)}
+\|\nabla f\|_{L^2(\R^2)}\Bigr)\nonumber\\
&\leq C\bigl(\|f\|_{L^2(B_1)}
+\|\nabla f\|_{L^2(\R^2)}\bigr).\label{Eq.L^2_loc}
\end{align}
Moreover there exists $R_0>1$ and $c_0>0$ such that $\int_{B(0,R_0)}\rho_0(x)\,dx\geq 4c_0 $. Let $R= R_0+2T\|\sqrt{\rho_0}u_0\|_{L^2}/\sqrt{4c_0}$. Then by Lemma \ref{lem2}, we have $\int_{B(0,R)}\rho(t,x)\,dx\geq c_0>0$, $\forall\ t\in[0, T] $. Then we get by Lemma \ref{lem_Lp} that
\begin{align}\label{f3}
 \|f\|_{L^2(B_1)}\leq\|f\|_{L^2(B(0,R))}\leq 
 C\bigl(\|\sqrt{\rho}(t, \cdot)f\|_{L^2(B(0,R))}+\|\nabla f\|_{L^2(B(0,R))}\bigr)\quad \forall\ t\in[0,T].
\end{align}
Then \eqref{lem4.weak_estimate_H1} follows from \eqref{Eq.L^2_loc} and \eqref{f3}.

Next we prove  \eqref{lem4.weak_estimate_H2}. Firstly, Sobolev embedding implies that 
\begin{align}\label{fLp}\|\sqrt{\rho}(t, \cdot)f\|_{L^p(\R^2)}\leq C \|\rho_0\|^{\f12}_{L^\infty}\|f\|_{H^1(\R^2)}.\end{align}
By Lemma \ref{lem2}, for $R= R_0+2T\|\sqrt{\rho_0}u_0\|_{L^2}/\sqrt{c_0},~c_0\in (0,\infty)$ we have
\begin{align*}
\int_{B(x_0,R)}\rho(t,x)\,dx\geq \f{c_0}{4}>0, \quad \forall\ x_0\in \R^2,\ \forall\ t\in[0, T].
\end{align*}
Then by Lemma \ref{lem_Lp}, for every $t\in[0,T]$, $\rho={\rho}(t, \cdot) $ we have
\begin{align*}
 \|f\|_{L^2(B(x_0,R))}\leq 
 C\bigl(\|\sqrt{\rho}f\|_{L^2(B(x_0,R))}+\|\nabla f\|_{L^2(B(x_0,R))}\bigr)\quad \text{~for all ~} x_0\in \R^2.
\end{align*}
Integrating it over $x_0\in \R^2$ gives 
\begin{align*}
 \int_{\R^2}\|f\|_{L^2(B(x_0,R))}^2\,dx_0\leq 
C\Bigl(\int_{\R^2}\|\sqrt{\rho}f\|_{L^2(B(x_0,R))}^2\,dx_0
+\int_{\R^2}\|\nabla f\|_{L^2(B(x_0,R))}^2\,dx_0\Bigr).
\end{align*}
By the Fubini theorem, we finally obtain
\begin{align}\label{fL2_control}
\|f\|_{L^2(\R^2)}\leq 
C\big(\|\sqrt{\rho}f\|_{L^2(\R^2)}+\|\nabla f\|_{L^2(\R^2)}\big),
\end{align}
which {along with \eqref{fLp}} implies  \eqref{lem4.weak_estimate_H2}.
\end{proof}

\section{Global existence and regularity of strong solution}\label{sec2}

This section is devoted to the global existence of strong solution and the regularity propagation of density patch and vacuum bubbles.

\subsection{A priori energy estimates}\label{uniform est}

Let $(\rho,u)$ be a smooth solution to the system \eqref{INS} on $[0,T]\times \R^2$ satisfying $0\leq \rho\leq \|\rho_0\|_{L^{\infty}}$.

First of all,  \eqref{energy} and Lemma \ref{lem1}--Lemma \ref{lem4} hold for strong solutions.
The following result has been proved in \cite{LSZ}, Lemma 3.2.

\begin{lem}\label{ nabla u}
There exists a positive constant $C$ depending only on $\|\rho_0\|_{L^{\infty}},~\|\sqrt{\rho_0}u_0\|_{L^2}$ and $\|\nabla u_0\|_{L^2}$ such that for $i=0,1$,
there holds
\begin{equation*}\begin{aligned}
&\sup_{t\in [0,T]} t^i\|\nabla u\|^2_{L^2}+\int_0^Tt^i\int_{\R^2}\rho|\dot{u}|^2\,dxdt
\leq C,\\
&\int_0^Tt^i\big(\|\nabla^2 u\|_{L^2}^2+\|\nabla P\|^2_{L^2}\big)\,dt\leq C.
\end{aligned}
\end{equation*}
\end{lem}

Motivated by Lemma 3.3 in \cite{LSZ},  we derive the estimates for the material derivative of the velocity, which are crucial for the higher order estimates of the velocity. 
Here we have an important observation $\p_iu^k\p_ku^j\p_ju^i=0$ due to an algebraic fact $\text{Tr}(A^3)=0$ if $\text{Tr}(A)=0$ for a $2\times 2$ matrix $A$. 
Then we can remove the assumption $\rho_0\in L^1(\R^2)$ in \cite{LSZ}.

\begin{lem}\label{dot u}
There exists a positive constant $C$ depending only on $\|\rho_0\|_{L^{\infty}}$, $\|\sqrt{\rho_0}u_0\|_{L^2}$ and $\|\nabla u_0\|_{L^2}$ such that for $i=1,2$,
 \begin{equation*}
 \begin{aligned}
&\sup_{t\in [0,T]} t^i\|\sqrt{\rho} \dot{u}\|^2_{L^2}
+\int_0^Tt^i\|\nabla\dot{u}\|_{L^2}^2\,dt\leq C,\\
&\sup_{t\in [0,T]} t^i\bigl(\|\nabla^2u\|^2_{L^2}+\|\nabla P\|^2_{L^2}\bigr)\leq C.
\end{aligned}\end{equation*}
\end{lem}

\begin{proof}
Applying the material derivative $D_t$ to the momentum equation of \eqref{INS}, we get \footnote{Repeated indices represent  summation.}
\begin{equation*}\begin{aligned}
\rho(\p_t\dot{u}+u\cdot\nabla \dot{u})-\Delta\dot{u}
=-\p_k(\p_ku\cdot \nabla u)-\p_k u\cdot \nabla \p_ku
-(\nabla P_t+u\cdot \nabla (\nabla P)).
\end{aligned}\end{equation*}
Taking the $L^2$ inner product with $\dot{u}$ to the above equation,  we get by integration by parts and $\dive u=0$ that 
\begin{equation}\begin{aligned}\label{Dt}
\f12 \frac{d}{dt}\int_{\R^2}\rho|\dot{u}|^2\,dx+\|\nabla \dot{u}\|^2_{L^2}
=\int_{\R^2}(\p_ku\cdot &\nabla u)\cdot\p_k\dot{u}\,dx+\int_{\R^2}\p_ku\cdot(\p_ku\cdot\nabla\dot{u})\,dx+J,
\end{aligned}\end{equation}
where
$$J\eqdef -\int_{\R^2}(\nabla P_t+u\cdot \nabla (\nabla P))\cdot\dot{u}\,dx.$$

The first two terms on the right hand side can be bounded by
\begin{equation}\begin{aligned}\label{e1}
\int_{\R^2}(\p_ku\cdot \nabla u)\cdot\p_k\dot{u}\,dx
+\int_{\R^2}\p_ku\cdot(\p_ku\cdot\nabla\dot{u})\,dx
\leq \f14 \|\nabla \dot{u}\|^2_{L^2}+C \|\nabla u\|^4_{L^4}.
\end{aligned}\end{equation}
As for $J$, using integration by parts, we get (as $\dive{u}=0$, $\dive\dot{u}=\p_iu^j\p_ju^i $)
\begin{align*}
J=&\int_{\R^2}(P_t\dive\dot{u}+(u\cdot \nabla P)\dive\dot{u}+{\p_iu^j\p_j P\dot{u}^{i}})\,dx\\
=&\int_{\R^2}(P_t+u\cdot \nabla P)\p_iu^j\p_ju^i\,dx
-\int_{\R^2} P\p_iu^j\p_j\dot{u}^i\,dx.
\end{align*}
We also have
\begin{equation*}\begin{aligned}
\int_{\R^2}(P_t+u\cdot \nabla P)\p_iu^j\p_ju^i\,dx
=\frac{d}{dt}\int_{\R^2} P\p_iu^j\p_ju^i\,dx
-\int_{\R^2} P D_t(\p_iu^j\p_ju^i)\,dx,
\end{aligned}\end{equation*}
and
\begin{align*}
-\int_{\R^2} P D_t(\p_iu^j\p_ju^i)\,dx
=&-\int_{\R^2} P(\p_i\dot{u}^j\p_ju^i+\p_iu^j\p_j\dot{u}^i)\,dx\\
&+\int_{\R^2} P(\p_iu^k\p_ku^j\p_ju^i+\p_iu^j\p_ju^k\p_ku^i)\,dx.
\end{align*}
Notice that $\p_iu^k\p_ku^j\p_ju^i=\p_iu^j\p_ju^k\p_ku^i=0$. 
Then by symmetry, we obtain
\begin{equation}\begin{aligned}\label{J}
J=\frac{d}{dt}\int_{\R^2} P\p_iu^j\p_ju^i\,dx
-3\int_{\R^2} P \p_iu^j\p_j\dot{u}^i\,dx.
\end{aligned}\end{equation}
Moreover, it follows from $\dive u=\text{\text{curl}}~\nabla\dot{u}=0$ that
\begin{equation}\begin{aligned}\label{P}
\left|\int_{\R^2} P \p_iu^j\p_j\dot{u}^i\,dx\right|
&\leq C\|P\|_{BMO}\|\p_iu^j\p_j\dot{u}^i\|_{\mathcal{H}^1}
\leq C\|\nabla P\|_{L^2}\|\nabla u\|_{L^2}\|\nabla \dot{u}\|_{L^2}.
\end{aligned}\end{equation}
Here $\mathcal{H}^1$ denotes the Hardy space. 

We denote
$$
\Psi(t)\eqdef \f12\int_{\R^2}\rho|\dot{u}|^2\,dx
-\int_{\R^2} P \p_iu^j\p_ju^i\,dx.$$
Plugging \eqref{e1}, \eqref{J} and \eqref{P} into \eqref{Dt} gives 
\begin{align}
    \Psi{'}(t)+\|\nabla \dot{u}\|^2_{L^2}{/2}
\leq C \|\nabla u\|^4_{L^4}+C \|\nabla P\|^2_{L^2}\|\nabla u\|^2_{L^2}.\label{psi}
\end{align}
Using the fact 
\begin{align*}
\left|\int_{\R^2} P\p_iu^j\p_ju^i\,dx\right|
\leq C\|P\|_{BMO}\|\p_iu^j\p_ju^i\|_{\mathcal{H}^1}
\leq C\|\nabla P\|_{L^2}\|\nabla u\|^2_{L^2}, 
\end{align*}
and the Stokes estimate
\begin{align}\label{stokes estimate}
\|\nabla^2u\|^2_{L^2}+\|\nabla P\|^2_{L^2}\leq C\|\rho\dot{u}\|^2_{L^2}
\leq C\|\sqrt{\rho}\dot{u}\|^2_{L^2},
\end{align}
we infer that
\begin{equation}\begin{aligned}\label{psi_2}
\f14\int_{\R^2}\rho|\dot{u}|^2\,dx-C\|\nabla u\|^4_{L^2}\leq
\Psi(t)\leq \int_{\R^2}\rho|\dot{u}|^2\,dx+C\|\nabla u\|^4_{L^2}.
\end{aligned}\end{equation}
By Gagliardo-Nirenberg inequality and \eqref{stokes estimate}, we have
\begin{equation}\begin{aligned}\label{u^4}
\|\nabla u\|^4_{L^4}+\|\nabla P\|^2_{L^2}\|\nabla u\|^2_{L^2}
&\leq C\|\nabla^2 u\|^2_{L^2}\|\nabla u\|^2_{L^2}
+\|\nabla P\|^2_{L^2}\|\nabla u\|^2_{L^2}\\
&\leq C\|\sqrt{\rho}\dot{u}\|^2_{L^2} \|\nabla u\|^2_{L^2}.
\end{aligned}\end{equation}

Multiplying \eqref{psi} by $t^i(i=1,2)$ and using \eqref{psi_2}-\eqref{u^4},  the first inequality  can be deduced from Lemma \ref{ nabla u} and \eqref{energy}. The second inequality is a direct consequence of the first inequality and \eqref{stokes estimate}.
\end{proof}

\subsection{$L^\infty$ bounds of $u$ and $\na u$}

In order to prove the existence of strong solution, we need to obtain a $L^\infty$ bound of $|u(t,x)|$. 

\begin{lem}\label{u_L^infty}
Assume \eqref{c0}. Then there exists a constant $C>0$ depending only on $\|\rho_0\|_{L^{\infty}}$,
$\|\sqrt{\rho_0}u_0\|_{L^2}$, $\|\nabla u_0\|_{L^2}$, $R_0$ and $c_0$ such that 
\begin{align*}
|u(t,x)|\leq C[\ln(2+|x-x_0|+t^{-1})]^{\f12}, \quad \forall\ t\in[0, T],\ \forall\ x\in\R^2.
\end{align*}
\end{lem}

\begin{proof}
Let $R(t):= R_0+2t\|\sqrt{\rho_0}u_0\|_{L^2}/\sqrt{c_0}$. Then by Lemma \ref{lem2}, we have 
\beno
\int_{B(x_0,R(t))} \rho(t,x)\,dx\geq c_0/4.
\eeno
So, we obtain
\begin{align*}
\frac{c_0}{4} \inf_{B(x_0,R(t))} |u|^2
\leq \int_{B(x_0,R(t))} \rho|u|^2\,dx\leq \|\sqrt{\rho_0}u_0\|^2_{L^2}\leq C,
\end{align*}
which implies that there exists a $x_t\in B(x_0,R(t))$ such that $|u(t,x_t)|\leq C_1$.

By Lemma \ref{ nabla u} and Lemma \ref{dot u}, we have
\begin{align*}
\|\nabla u(t)\|^2_{L^2}=\int_{\R^2}|\hat{u}(t, \xi)|^2|\xi|^2\,d\xi\leq C(1+t)^{-1},\\
\|\nabla^2 u(t)\|^2_{L^2}=\int_{\R^2}|\hat{u}(t, \xi)|^2|\xi|^4\,d\xi\leq Ct^{-1}(1+t)^{-1},
\end{align*}
which ensure that 
\begin{align}\label{ut}
\int_{\R^2}|\hat{u}(t,\xi)|^2|\xi|^2(1+t|\xi|^2)\,d\xi\leq C(1+t)^{-1}.    
\end{align}
Thus,  for any $x\neq y\in \R^2$, we have (as $|e^{i\xi\cdot x}-e^{i\xi\cdot y}|\leq\min ( |\xi||x-y|,2) $)
\begin{align}
    |u(t, x)-u(t, y)|
&\leq \int_{\R^2}|\hat{u}(t, \xi)||e^{i\xi\cdot x}-e^{i\xi\cdot y}|\,d\xi\nonumber\\
&\leq \biggl(\int_{\R^2}|\hat{u}(t, \xi)|^2|\xi|^2(1+t|\xi|^2)\,d\xi\biggr)^{1/2}
\biggl(\int_{\R^2}\frac{\min ( |\xi|^2|x-y|^2,4)}{|\xi|^2(1+t|\xi|^2)}\,d\xi\biggr)^{1/2},\label{u_x_y}
\end{align}
where 
\begin{align*}
\int_{\R^2}\frac{\min ( |\xi|^2|x-y|^2,4)}{|\xi|^2(1+t|\xi|^2)}\,d\xi&=\biggl(\int_{|\xi|\leq |x-y|^{-1}}+\int_{|\xi|> |x-y|^{-1}}\biggr)
\frac{\min ( |\xi|^2|x-y|^2,4)}{|\xi|^2(1+t|\xi|^2)}\,d\xi\\
&\leq C\ln \left(2+\frac{|x-y|}{\sqrt{t}}\right),
\end{align*}
 which along with \eqref{ut} and \eqref{u_x_y} ensures that
\begin{equation}\label{Eq.u(x)-u(y)}
|u(t, x)-u(t, y)|\leq C(1+t)^{-\f12}\left[\ln \left(2+|x-y|/\sqrt{t}\right)\right]^{\f12}.
\end{equation}

Therefore, by \eqref{Eq.u(x)-u(y)}, $x_t\in B(x_0,R(t))$, $|u(t,x_t)|\leq C_1$ and $R(t)= R_0+2t\|\sqrt{\rho_0}u_0\|_{L^2}/\sqrt{c_0}$, we arrive at 
\begin{equation*}\begin{aligned}
|u(t,x)|&\leq |u(t,x_t)|+|u(t,x)-u(t,x_t)|\\
&\leq C+C(1+t)^{-\f12}[\ln \bigl(2+t^{-1}+|x-x_0|+R(t)\bigr)]^{\f12}\\
&\leq C[\ln (2+|x-x_0|+t^{-1})]^\f12.
\end{aligned}\end{equation*}

This completes the proof of the lemma.
\end{proof}

For the estimate of $\|\nabla u\|_{L^{\infty}}$, we need the following lemma.
\begin{lem}\label{Lf}
Let $f:\R^2\to \R$ and $F:\R^2\to \R^2$ be smooth functions satisfying \footnote{ $\partial_zf=\frac12(\partial_1f-i\partial_2f)$, $\pa_{\bar z}f=\frac12(\partial_1f+i\partial_2f)$.}
$$Lf=\text{curl}~F, \quad\text{~where~}L\in\{\Delta,~\p^2_z,~\p^2_{\bar{z}}\},$$
then there exists a constant $C>0$ independent of $f,F$ such that
\begin{align*}
|f(x)|\leq Cr^{-1}\|f\|_{L^2(B(x,r))}
+C\int_{B(x,r)}\frac{|F(y)|}{|y-x|}\,dy, \quad \forall~x\in\R^2,~\forall~ r>0 .   
\end{align*}
\end{lem}

\begin{proof}
By the translation and dilation, it suffices to show that 
\begin{align}\label{Eq.f(0)_estimate}
|f(0)|\leq C\|f\|_{L^2(B_1)}+ C\int_{B_1}\frac{|F(y)|}{|y|}\,dy     
\end{align}
for some constant $C>0$ independent of $f$ and $F$. We split the proof into three cases.\smallskip

{\bf Case 1. $L=\Delta$. } Let $\eta_0\in \mathcal{C}_c^{\infty}(\R; [0,1])$ be such that $\eta_0|_{[-1/2, 1/2]}=1$ and  $\operatorname{supp}\eta_0\subset[-1,1]$, and let $\xi_0\in \mathcal{C}_c^{\infty}(\R^2;[0,1])$ be given by $\xi_0(x)=\eta_0(|x|^2)$ for all $x\in\R^2$. 
Thanks to 
\begin{align*}
\Delta (\xi_0f)=2\dive (\nabla \xi_0f)-\Delta \xi_0f+\text{curl} (\xi_0 F) +F\cdot \nabla^\perp \xi_0,
\end{align*}
we have
\begin{equation*}\begin{aligned}
f(0)=\int_{\R^2}K(-y)\bigl[
2\dive (\nabla \xi_0f)-\Delta \xi_0f+\text{curl} (\xi_0 F) +F\cdot \nabla^\perp \xi_0\bigr](y)\,dy,    
\end{aligned}\end{equation*}
where $K(x)=\frac{1}{2\pi}\ln|x|$.

Using the integration by parts and support properties of $\xi_0$, we get
\begin{align*}
    &\left|\int_{\R^2}K(-y)\dive (\nabla \xi_0 f)(y)\,dy\right|+\left|\int_{\R^2}K(-y)(\Delta \xi_0 f)(y)\,dy\right|\\
    =&\left|\int_{B_1\setminus B_{1/2}}\nabla (K(-y))\cdot(\nabla\xi_0 f)(y)\,dy\right|+\left|\int_{B_1\setminus B_{1/2}} K(-y)(\Delta \xi_0 f)(y)\,dy\right|\\
    \leq & \ C\|f\|_{L^1(B_1)}\leq C\|f\|_{L^2(B_1)},
\end{align*}
and
\begin{align*}
    &\left|\int_{\R^2}K(-y)\text{curl} (\xi_0 F)(y)\,dy\right|+\left|\int_{\R^2}K(-y)(F\cdot \nabla^\perp \xi_0)(y)\,dy\right|\\
    =&\left|\int_{B_1}\nabla^\perp (K(-y))\cdot(\xi_0 F)(y)\,dy\right|+\left|\int_{B_1\setminus B_{1/2}}K(-y)(F\cdot \nabla^\perp\xi_0)(y)\,dy\right|
    \leq C\int_{B_1}\frac{|F(y)|}{|y|}\,dy.
\end{align*}
Combining above estimates, we obtain \eqref{Eq.f(0)_estimate}.\smallskip

{\bf Case 2. $L=\p^2_z$.}
Notice that  $\frac{z}{\bar{z}}\p^2_z f=\p_z(\frac{z}{\bar{z}}\p_z f-\frac{f}{\bar{z}})$. For $\varepsilon\in(0,1)$, we let 
\[\eta_1(s):=\eta_0(s/\varepsilon^2),\quad \eta(s):=\eta_0(s)-\eta_1(s),\quad\forall\ s\in\R.\]
Firstly, we have 
\begin{align}
 \label{I0}   &\int_{|z|\leq 1} \eta(|z|^2)\frac{z}{\bar{z}}\p^2_z f \,dz= -\int_{|z|\leq 1} \p_z(\eta(|z|^2))\left(\frac{z}{\bar{z}}\p_z f-\frac{f}{\bar{z}}\right)\,dz=I_1+I_2+I_3,\quad\text{where}\\
 \notag   &I_1=\int_{|z|\leq 1} f\p_z\left(\frac{z}{\bar{z}}\p_z\eta_0(|z|^2)\right)+\frac{f}{\bar{z}}\p_z\eta_0(|z|^2)\,dz,\\
\notag&I_2=\int_{|z|\leq \epsilon} \p_z\eta_1(|z|^2)\frac{z}{\bar{z}}\p_z f\,dz, \quad I_3=-\int_{|z|\leq \epsilon} \frac{f}{\bar{z}}\p_z\eta_1(|z|^2)\,dz.
\end{align}
Since $\p_z \eta_j(|z|^2)=\eta_j'(|z|^2)\bar{z}$ for $j=0,1$, then 
\begin{align*}
f\p_z\left(\frac{z}{\bar{z}}\p_z\eta_0(|z|^2)\right)
+\frac{f}{\bar{z}}\p_z\eta_0(|z|^2)
= \bigl(2\eta'_0(|z|^2)+|z|^2 \eta_0''(|z|^2)\bigr)f,
\end{align*}
so we get
\begin{align}\label{I1}
|I_1|\leq C\|f\|_{L^1(B_1)}\leq C\|f\|_{L^2(B_1)},    
\end{align}
and similarly
\begin{equation}\label{I2}
|I_2|=\left|\int_{|z|\leq \epsilon} \eta'_1(|z|^2)z\p_z f\,dz\right|
\leq \epsilon \|\nabla f\|_{L^{\infty}(B_\varepsilon)}\int_{|z|\leq \epsilon}|\eta'_1(|z|^2)|\,dz\leq C\epsilon \|\nabla f\|_{L^{\infty}(B_{\epsilon})}.
\end{equation}
Note that  $\int_{|z|\leq \epsilon}\eta'_1(|z|^2)dz=-\pi$, then
\begin{align}
|I_3+\pi f(0)|&=\left|\int_{|z|\leq \epsilon} \eta'_1(|z|^2) (f(z)-f(0)) \,dz\right|\nonumber\\
&\leq \epsilon \|\nabla f\|_{L^{\infty}(B_{\epsilon})}
\int_{|z|\leq \epsilon}|\eta'_1(|z|^2)|\,dz
\leq C{\epsilon} \|\nabla f\|_{L^{\infty}(B_{\epsilon})}.\label{I3}
\end{align}
On the other hand, using integration by parts, we have
\begin{align}
\left|\int_{|z|\leq 1} \eta(|z|^2)\frac{z}{\bar{z}}\p^2_z f \,dz\right|
&= \left|\int_{|z|\leq 1} \eta(|z|^2)\frac{z}{\bar{z}}\text{curl} F \,dz\right|\nonumber \\
&=\left|\int_{|z|\leq 1} F\cdot\nabla^\perp\left(\eta(|z|^2)\frac{z}{\bar{z}}\right)\,dz\right|\leq C\int_{|z|\leq 1} \frac{|F(z)|}{|z|}\,dz,\label{F}
\end{align}
where we have used the fact that 
\begin{align*}
\left|\p_z\left(\eta(|z|^2)\frac{z}{\bar{z}}\right)\right|+\left|\p_{\bar{z}}\left(\eta(|z|^2)\frac{z}{\bar{z}}\right)\right|=
\left|\eta'(|z|^2)z+\eta(|z|^2)\frac{1}{\bar{z}}\right|+\left|\eta'(|z|^2)\frac{z^2}{\bar{z}}-\eta(|z|^2)\frac{z}{\bar{z}^2}\right|\leq \frac{C}{|z|}.
\end{align*}
Combining \eqref{I0}, \eqref{I1}, \eqref{I2}, \eqref{I3}, \eqref{F} and letting $\epsilon\to 0^+$, we obtain \eqref{Eq.f(0)_estimate}.\smallskip

{\bf Case 3. $L=\p^2_{\bar{z}}$.} Notice that $\frac{\bar{z}}{z}\p^2_{\bar{z}} f=\p_{\bar{z}}(\frac{\bar{z}}{z}\p_{\bar{z}} f-\frac{f}{z})$. The proof is similar to Case 2.
\end{proof}

Now we are in a position to give a priori estimate of $\|\nabla u\|_{L^1(0,T;L^{\infty})}$.

\begin{lem}\label{Lip}
Assume that $0\le \rho\le C_0$ and $u$ is a smooth solution of the following equation 
\begin{align}\label{rho_u_curl}
\rho \dot{u}-\Delta u+\nabla P=0,\quad \dive u=0.    
\end{align}
Then there is a constant $C>0$ depending only on $C_0$, $\|\sqrt{\rho_0}u_0\|_{L^2}$ and $\|\nabla u_0\|_{L^2}$ such that
\begin{align}\label{nabla_u}
|\nabla u(t,x)|\leq& Cr^{-1}\|\nabla u(t)\|_{L^2}
+C\int_{B(x,r)}\frac{|\rho \dot{u}(t,y)|}{|y-x|}\,dy,\quad\forall\ t\in[0, T],\ \forall\ x\in\R^2,\ \forall\ r>0,\\
&\text{and}\qquad \int_0^T \|\nabla u(t)\|_{L^{\infty}} \,dt\leq C[\ln(1+T)]^{\f12}.\label{lip}
\end{align}
\end{lem}

\begin{proof}
Let $u=(-\p_2\psi,\p_1\psi)$, then  $\text{curl}~u=\Delta \psi$. 
By \eqref{rho_u_curl}, we have $\text{curl}~(\Delta u)=\text{curl}~ (\rho\dot{u})$, and thus $\Delta^2 \psi=\text{curl}~(\rho\dot{u})$. Note that 
\beno
|\nabla u|\sim C(|\psi_1|+|\psi_2|+|\psi_3|),
\eeno
 where $\psi_1:=\Delta \psi$, $\psi_2:=\p_z^2\psi$ and $\psi_3:=\p_{\bar{z}}^2\psi$.
Hence, by the fact $\Delta^2 \psi=\text{curl}~(\rho\dot{u})$ and $\Delta =4\p_z\p_{\bar{z}}$, we obtain
\begin{align*}
 \Delta \psi_1=\text{curl}~(\rho\dot{u}),\quad 
 16\p^2_{\bar{z}}\psi_2=\text{curl}~(\rho\dot{u}), \quad 16\p^2_z\psi_3=\text{curl}~(\rho\dot{u}).
\end{align*}
{As $\|\psi_k\|_{L^2} 
 \leq C\|\nabla u\|_{L^2}$ for $k=1,2,3$,}
then by Lemma \ref{Lf}, we get the desired estimate \eqref{nabla_u}. \smallskip

Next we prove \eqref{lip}. We claim that there exists a constant $C>0$ such that
\begin{equation}\label{Eq.nabla_u_claim}
    \|\nabla u\|_{L^{\infty}}
\leq Cr_1^{-1}\|\nabla u\|_{L^2}+Cr_1C_0\|\nabla \dot{u}\|_{L^2}
+CC_0^{\f12}\|\sqrt{\rho}\dot{u}\|_{L^2}
\end{equation}
for any $r_1>0$. Assume \eqref{Eq.nabla_u_claim} for the moment. Let $r_1=C_0^{-1/2}\|\nabla u\|^{1/2}_{L^2}\|\nabla \dot{u}\|^{-1/2}_{L^2}$.
Then we infer that
\if0\begin{align*}
\|\nabla u\|_{L^{\infty}}
\leq CC_0^{\f12}[\ln (3+\|\nabla u\|_{L^2}
\|\nabla \dot{u}\|_{L^2}/\|\sqrt{\rho}\dot{u}\|^2_{L^2})]^{\f12}
\|\sqrt{\rho}\dot{u}\|_{L^2},
\end{align*}
which along with the fact $\ln(1+s)\leq s$ for all $s>0$ gives \fi
\begin{align*}
    \|\nabla u\|_{L^{\infty}}\leq C\|\nabla u\|_{L^2}^{1/2}\|\nabla \dot{u}\|_{L^2}^{1/2}+C\|\sqrt{\rho}\dot{u}\|_{L^2}.
\end{align*}
Now it follows from \eqref{energy}, Lemma \ref{ nabla u}, Lemma \ref{dot u} and H\"{o}lder's  inequality that
\begin{align*}
    \int_0^T\|\nabla u\|_{L^{\infty}}\,dt&\leq \int_0^T \|\nabla u\|^{\f12}_{L^{2}}\|\nabla \dot{u}\|^{\f12}_{L^2} \,dt+\int_0^T \|\sqrt{\rho} \dot{u}\|_{L^2} \,dt\\
    &\leq \left(\int_0^T \|\nabla u\|^2_{L^{2}} \,dt\right)^{1/4}\left(\int_0^T t(1+t)\|\nabla \dot{u}\|^2_{L^{2}}\, dt\right)^{1/4}\left(\int_0^T t^{-\f12}(1+t)^{-\f12}\,dt\right)^{1/2}\\
    &\qquad\qquad+\left(\int_0^T (1+t)\|\sqrt{\rho} \dot{u}\|^2_{L^2}\,dt\right)^{1/2}\left(\int_0^T (1+t)^{-1}\,dt\right)^{1/2}\\
    &\leq C[\ln(1+T)]^{\f12}.
\end{align*}
This proves \eqref{lip}.

Finally, it remains to prove \eqref{Eq.nabla_u_claim}. Fix $x\in\R^2$ and $r_1>0$. We emphasize that in the proof of \eqref{Eq.nabla_u_claim}, the constant $C$ is independent of $x$ and $r_1$. 
Let $\avg(\dot{u})=\frac{1}{|B(x,r_1)|}\int_{B(x,r_1)}\dot{u}(y)\,dy$. By Poincar\'{e} inequality we get that 
\begin{align}\label{Eq.Poincare}
\|\dot{u}-\avg(\dot{u})\|_{L^m(B(x,r_1))}
\leq C_mr_1^{\frac{2}{m}}\|\nabla \dot{u}\|_{L^2(B(x,r_1))},\quad\forall\ m\in[1, +\infty).
\end{align}
As a consequence, for all $p>2$, if we let $p'\in(2, +\infty)$ be such that $\frac{1}{p}+\frac{1}{p'}=1$, then
\begin{equation}\begin{aligned}\label{dot_u_minus}
\int_{B(x,r_1)}\frac{\rho|\dot{u}-\avg(\dot{u})|}{|y-x|}\,dy
&\leq C_0 \|\dot{u}-\avg(\dot{u})\|_{L^p(B(x,r_1))}
\||\cdot|^{-1}\|_{L^{p'}(B(0,r_1))}\\
&\leq Cr_1C_0\|\nabla \dot{u}\|_{L^2}.
\end{aligned}
\end{equation}
On the other hand, as $|\cdot|^{-1}\in L^{2,\infty}(\R^2)$ ,  we get by H\"{o}lder inequality and interpolation inequality in the Lorentz space that 
\begin{align*}
\int_{B(x,r_1)}\frac{\rho (t,y)}{|y-x|}\,dy
&\leq \|\rho\|_{L^{2,1}(B(x,r_1))} \||\cdot|^{-1} \|_{L^{2,\infty}(\R^2)}
\leq C\|\rho\|^{1/2}_{L^{\infty}}\|\rho\|^{1/2}_{L^1(B(x,r_1))},
\end{align*}
and by \eqref{Eq.Poincare}, we have
\begin{equation*}\begin{aligned}
&\|\rho\|_{L^1(B(x,r_1))}|\avg(\dot{u})|^2=\int_{B(x,r_1)}\rho (y)|\avg(\dot{u})|^2\,dy
\\ \leq& 2\int_{B(x,r_1)}\rho|\dot{u}-\avg(\dot{u})|^2\,dy
+2\int_{B(x,r_1)}\rho|\dot{u}|^2\,dy
\leq Cr^2_1C_0\|\nabla \dot{u}\|^2_{L^2}+C\|\sqrt{\rho}\dot{u}\|^2_{L^2},
\end{aligned}\end{equation*}
hence
\begin{equation}\begin{aligned}\label{bar_dot_u}
\int_{B(x,r_1)}\frac{\rho (y)|\avg(\dot{u})|}{|y-x|}\,dy
&\leq C\|\rho\|^{1/2}_{L^{\infty}}\|\rho\|^{1/2}_{L^1(B(x,r_1))}|\avg(\dot{u})|\\
&\leq Cr_1C_0\|\nabla \dot{u}\|_{L^2}
+CC_0^{\f12}\|\sqrt{\rho}\dot{u}\|_{L^2}.
\end{aligned}\end{equation}
Thanks to \eqref{dot_u_minus} and \eqref{bar_dot_u}, we have
\begin{align}\label{r2}
\int_{B(x,r_1)}\frac{|\rho \dot{u}(t,y)|}{|y-x|}\,dy
\leq Cr_1C_0\|\nabla \dot{u}\|_{L^2}+CC_0^{\f12}\|\sqrt{\rho}\dot{u}\|_{L^2}.
\end{align}
By \eqref{r2}, we deduce \eqref{Eq.nabla_u_claim} from \eqref{nabla_u}.
\end{proof}

\subsection{Proof of global existence}

In this subsection, we prove Theorem \ref{thm:strong}. We follow the argument in \cite{DM}. The  idea is to take advantage of classical result to construct smooth solutions corresponding to smoothed-out approximate data with no vacuum, then to pass to the limit. More precisely, we consider
$$ u_0^{\epsilon}\in \mathcal{C}^{\infty}(\R^2) \with  \dive u_0^{\epsilon}=0 \andf \rho_0^{\epsilon}\in \mathcal{C}^{\infty}(\R^2) \with \epsilon\leq \rho_0^{\epsilon}\leq C_0 $$
such that
\begin{equation*}\begin{aligned}
& u_0^{\epsilon}\to u_0 \text{~in~} H^1, \quad \rho_0^{\epsilon}\rightharpoonup \rho_0 \text{~in~} L^{\infty} \text{~weak-*},
\andf~ \rho_0^{\epsilon}\to \rho_0 \text{~in~} L^{p}_{\text{loc}} \text{~if~} p<\infty.
\end{aligned}\end{equation*}As $\rho_0\not\equiv 0$, there exist $R_0\in (0,\infty)$ and $c_0>0$ such that $\int_{B(0,R_0)} \rho_0\,dx> c_0>0$. Then $\int_{B(0,R_0)} \rho_0^{\epsilon}\,dx> c_0>0$ for $\epsilon $ small enough, thus \eqref{c0} holds for  $x_0=0$ and $\rho_0$ replaced by $\rho_0^{\epsilon} $. We also have $\|u_0^{\epsilon}\|_{L^2}\leq \|u_0\|_{L^2}+1$ for $\epsilon $ small enough, and $\|\sqrt{\rho_0^{\epsilon}}u_0^{\epsilon}\|_{L^2}\leq\|\rho_0^{\epsilon}\|_{L^{\infty}}^{1/2}\|u_0^{\epsilon}\|_{L^2}\leq C_0^{1/2}(\|u_0\|_{L^2}+1)$.
In light of the classical strong solution theory for the system \eqref{INS}, there exists a unique global smooth solution $(\rho^{\epsilon}, u^{\epsilon}, P^{\epsilon})$
corresponding to data $(\rho_0^{\epsilon},u_0^{\epsilon})$. Thus, the triple $(\rho^{\epsilon}, u^{\epsilon}, P^{\epsilon})$  satisfies all the a priori estimates of previous subsections uniformly with respect to $\epsilon$. In particular, by Lemma \ref{ nabla u}, Lemma \ref{dot u} and Lemma \ref{u_L^infty}, we have $|u^{\epsilon}(t,x)|\leq C[\ln(2+|x|+t^{-1})]^{\f12}$, $\int_0^T\|u^{\epsilon}\|^2_{L^{\infty}(B_R)}\leq C(R,T) $ and
\begin{equation*}\begin{aligned}
\int_0^T t(\|\nabla \dot{u^{\epsilon}}\|^2_{L^2}
+\|\nabla u^{\epsilon}\|^2_{L^2}\|\nabla^2u^{\epsilon}\|_{L^2}^2+\|u^{\epsilon}\|^2_{L^{\infty}(B_R)}\|\nabla^2 u^{\epsilon}\|^2_{L^2})\,dt\leq C(R,T).
\end{aligned}\end{equation*}
By the definition, we have
\begin{equation*}\begin{aligned}
\|\nabla u^{\epsilon}_t\|_{L^2(B_R)}
&\leq \|\nabla \dot{u^{\epsilon}}\|_{L^2(B_R)}
+\|\nabla u^{\epsilon}\cdot\nabla u^{\epsilon}
+u^{\epsilon}\cdot \nabla^2 u^{\epsilon}\|_{L^2(B_R)}\\ 
&\leq \|\nabla \dot{u^{\epsilon}}\|_{L^2}
+C(\|\nabla u^{\epsilon}\|_{L^2}
+\|u^{\epsilon}\|_{L^{\infty}(B_R)})\|\nabla^2 u^{\epsilon}\|_{L^2}
\end{aligned}\end{equation*}
for any $R>0$. Thus, we obtain
\begin{equation*}\begin{aligned}
\int_0^T t\|\nabla u_t^{\epsilon}\|^2_{L^2(B_R)}\,dt
\leq C(R,T).
\end{aligned}
\end{equation*}
Similarly, we have 
\begin{equation*}\begin{aligned}
\int_0^T \|\sqrt{\rho^{\epsilon}}u_t^{\epsilon}\|^2_{L^2(B_R)}\,dt
\leq C(R,T).
\end{aligned}
\end{equation*}
Let $R(T)=R_0+2TC_0^{1/2}(\|u_0\|_{L^2}+1)/\sqrt{c_0}$ then $R(T)\geq R_0+2T\|\sqrt{\rho_0^{\epsilon}}u_0^{\epsilon}\|_{L^2}/\sqrt{c_0}$ 
for $\epsilon $ small enough. By Lemma \ref{lem2} we have $\int_{B(0,R)} \rho^{\epsilon}(t,x)\,dx\geq c_0/4 $ for $R\geq R(T)$, $t\in[0,T]$. 
Then by Lemma \ref{lem_Lp} we have $\| u^{\epsilon}_t\|_{L^2(B_R)}\leq C_R(\|\nabla u_t^{\epsilon}\|_{L^2(B_R)}
+\|\sqrt{\rho^{\epsilon}}u_t^{\epsilon}\|_{L^{2}(B_R)}) $ for $R\geq R(T)$, $t\in[0,T]$.
Thus $\int_0^T t\| u_t^{\epsilon}\|^2_{L^2(B_R)}\,dt
\leq C(R,T)$ for $R\geq R(T)$, which also holds for $0<R<R(T)$ as the left hand side is increasing in $R$.

By Lemma 3.4 in \cite{DM}, we know that for all $\alpha\in (0,1/2)$
\begin{align}\label{inte}
\| u^{\epsilon}\|_{H^{1/2-\alpha}(0,T;H^1(B_R))}
\leq \| u^{\epsilon}\|_{L^2(0,T;H^1(B_R))}
+C_{\alpha,T}\|t^{1/2} u^{\epsilon}_t\|_{L^2(0,T;H^1(B_R))},    
\end{align}
where $C_{\alpha,T}$ depending only on $\alpha$ and $T$. By \eqref{energy} and Lemma \ref{ nabla u}, we also have
\begin{align*}
\|\nabla u^{\epsilon}\|_{L^2(0,T;{H}^1)}\leq C,   
\end{align*}
By the interpolation with \eqref{inte}, we can deduce that  
\begin{align*}
\|(u^{\epsilon},\nabla u^{\epsilon})\|_{{H}^{1/2-\al}([0,T]\times B_R)} \leq C(R,T)\quad   \text{~ for all~} R,T\in(0,+\infty) .   
\end{align*}
This implies that up to a subsequence, $u^{\epsilon}\to u$ in $L^2_{loc}([0,+\infty);{H}^{1}_{loc})$. By {Theorem} 2.5 in \cite{PL}, one can show that 
\if0\begin{align*}
(\rho^{\epsilon})^2\rightharpoonup (\rho)^2 \text{~in~} 
L^{\infty}(\R^{+};L^{\infty}(\R^2)), 
\end{align*}
which eventually implies that\fi 
\begin{align*}
\rho^{\epsilon}\to \rho\text{~in~} \mathcal{C}([0,T];L^p(B_R)) \text{~for all~} 1\leq p<+\infty,\ T,R\in(0,+\infty). 
\end{align*}
Thus, we can pass to the limit for all the terms in the definition of weak solutions, and the weak energy inequality can be obtained by Fatou property. Moreover Lemma \ref{ nabla u}, Lemma \ref{dot u}, \eqref{lip} and Lemma \ref{u_L^infty} also hold for the limit solution $(\rho,u)$.
Furthermore, the moment equation is fulfilled in the strong sense, i.e.,
\begin{align*}
\rho(\partial_tu+u\cdot\nabla u)-\Delta u+\nabla P=0 \quad\text{~in~} L^2_{loc}(\R^{+};L^2_{loc}(\R^2))   
\end{align*}
for some pressure function $P$ with the estimate in Lemma \ref{ nabla u}.

\subsection{Regularity propagation of density patch and vacuum bubbles}\label{sec3}

\begin{proof}[Proof of Theorem \ref{nabla^2 u_Lm}]
    We assume that $\rho_0$ satisfies (H1) or (H2). By the Stokes estimate and Lemma \ref{lem4}, we have
    \begin{align*}
        \|\sqrt{t}\nabla^2u\|_{L^2([0,T];L^m(\R^2))}&\leq C\|\sqrt{t}\rho\dot{u}\|_{L^2([0,T];L^m(\R^2))}\leq C\|\sqrt{t}\sqrt\rho\dot{u}\|_{L^2([0,T];L^m(\R^2))}\\
        &\leq C\|(\sqrt {t\rho}\dot{u}, \sqrt t\nabla \dot u)\|_{L^2([0,T];L^2(\R^2))}\leq C,
    \end{align*}
    where we have used Lemma \ref{dot u} in the last inequality, and $C>0$ is a constant depending only on $\rho_0, u_0$, $m\geq 2$ and $T>0$. \end{proof}

Now we prove Corollary \ref{cor1}. 

\begin{proof}
Assume that $D$ corresponds to the level set $\{f_0>0\}$ of some function $f_0:\R^2\to \R_{\geq 0}$ with $C^{1,\alpha}$ regularity and $df_0\ne 0$ on $\{f_0=0\}$.
Then we have $D_t=f_t^{-1}(0,\infty)$ with $f_t\eqdef f_0\circ X(t,0,\cdot)$, and $df_t\ne0$ on $\{f_t=0\}$.

Next we check that $\rho_0$ satisfies (H1) or (H2). In fact, if $a=0$, then $\rho_0=b\mathbf{1}_{D}$,  which implies that $\rho_0$ satisfies (H1) as $D$ is bounded. If $a>0$, then $\rho_0\geq a(1-\mathbf{1}_{D})$, and thus there exist $R_0>0, c_0>0$ such that for all $x_0\in\R^2$ we have 
$$\int_{B(x_0,R_0)}\rho_0 \,dx\geq
\int_{B(x_0,R_0)}a \,dx-\int_{\R^2}a\mathbf{1}_{D}\,dx=a\pi R_0^2-a{|D|}=c_0>0,$$
which means that $\rho_0$ satisfies (H2).

By Theorem \ref{thm:strong}  and Theorem \ref{nabla^2 u_Lm}, there exists a global solution $(\rho, u,\nabla P)$ satisfying $\nabla u\in L^{\infty}(\R^+;L^2(\R^2))$ and $\sqrt{t}\nabla^2u\in L^2([0,T];L^m(\R^2))$ for all $T>0$, $m\geq 2$.  For $\alpha\in(0,1)$, we have
\begin{align*}
&[\nabla u]_{\alpha}\leq C\|\nabla u\|^{1-s}_{L^2}\|\nabla^2u\|^s_{L^m}\quad \text{~~for~} m>\frac{2}{1-\alpha},~s=\frac{1+\alpha}{2-2/m}\in(0,1),\\
&\|\nabla u\|_{L^{\infty}}\leq 
C\|\nabla u\|^{1-s'}_{L^2}\|\nabla^2u\|^{s'}_{L^m} 
\quad\text{~~for~} s'=\frac{1}{2-2/m}\in(0,s),
\end{align*}
which imply that $t^{s/2}\nabla u\in L^2([0,T];C^{0,\alpha})$ and $\nabla u\in L^1([0,T];C^{0,\alpha})$.
Consequently, the flow $X(t,0,\cdot)$ is in $C^{1,\alpha}$, which implies that $f_t$ is in $C^{1,\alpha}$.
\end{proof}

\section{Weak-strong uniqueness }\label{sec4}

Throughout this section, let $(\bar{\rho}, \bar{u})$ be the strong solution obtained in Theorem \ref{thm:strong} and $(\rho, u)$ be the weak solution provided by Lions in \cite{PL} with the same initial data $(\rho_0, u_0)$. Furthermore, we assume that  $\rho_0$ satisfies (H1) or (H2).

\subsection{Trilinear estimates}    

Let $j\in C_c^{\infty}(\R^2;[0,9])$ be fixed such that $\operatorname{supp}j\subset B_{1/4}$, $\int_{\R^2}j(x)dx=1$ and we denote $j_{\epsilon}(x):=\epsilon^{-2}j(x/{\epsilon})$ for $\varepsilon>0$.

\begin{lem}\label{uni_H^-1_bubble}
There exists $C>0$ such that for any
$f, f_1, f_2\in H^1,~g\in \dot{H}^{-1}\cap L^{\infty}$ with $\|g\|_{\dot{H}^{-1}}\leq 1/2$, and any $\epsilon\in(0, 1/2]$ we have
\begin{align}\label{H-1_uni_bubble}
\biggl|\int_{\R^2} gf^2\,dx\biggr|&\leq C\big(\|g\|_{\dot{H}^{-1}}|\ln \epsilon|^{1/2}+\epsilon\|g\|_{L^{\infty}}\big)\|f\|^2_{H^1},\\
\biggl|\int_{\R^2} gf_1f_2\,dx\biggr|&\leq C\big(\|g\|_{\dot{H}^{-1}}|\ln \epsilon|^{1/2}+\epsilon\|g\|_{L^{\infty}}\big)\|f_1\|_{H^1}\|f_2\|_{H^1}.\label{Eq.bubble2}
\end{align}
\end{lem}

\begin{proof}
For any $\epsilon\in(0,1/2]$, we have
\begin{align*}
\Bigl|\int_{\R^2}gf^2\,dx\Bigr|&=\Bigl|\int_{\R^2}g\bigl((j_{\epsilon}*f)^2+(f-j_{\epsilon}*f)^2
+2(j_{\epsilon}*f)(f-j_{\epsilon}*f)\bigr)\,dx\Bigr|\\
&\leq \|g\|_{\dot{H}^{-1}}\|(j_{\epsilon}*f)^2\|_{\dot{H}^{1}}
+\|g\|_{L^{\infty}}\|f-j_{\epsilon}*f\|^2_{L^2}
+2\|g\|_{L^{\infty}}\|j_{\epsilon}*f\|_{L^2}\|f-j_{\epsilon}*f\|_{L^2}\\
&\leq C\|g\|_{\dot{H}^{-1}}\|j_{\epsilon}*f\|_{L^{\infty}}
\|j_{\epsilon}*f\|_{\dot{H}^{1}}
+C{\epsilon}^2\|g\|_{L^{\infty}}\|\nabla f\|^2_{L^2}
+C\epsilon\|g\|_{L^{\infty}}\|f\|^2_{H^1}\\
&\leq C|\ln \epsilon|^{1/2}\|g\|_{\dot{H}^{-1}}\|f\|^2_{H^1}
+C\epsilon\|g\|_{L^{\infty}}\|f\|^2_{H^1},
\end{align*}
where we have used the estimates $\|j_{\epsilon}*f\|_{L^{\infty}}\leq \|\hat{j_{\epsilon}}\hat{f}\|_{L^1}\leq C|\ln \epsilon|^{1/2}\|f\|_{{H}^1}$ and {(as $\int_{\R^2} j(x)=1$)}
\begin{equation}\label{Eq.f-conv_L^2}
    \|f-j_{\epsilon}*f\|_{L^{2}}\leq C\epsilon\|f\|_{\dot{H}^1}.
\end{equation}
This proves \eqref{H-1_uni_bubble}. The proof of \eqref{Eq.bubble2} is similar.
\end{proof}

\begin{lem}\label{uni_H^-1_patch}
For any $C_0>0$, there exists $C>0$ such that for any
$f, f_1, f_2\in \dot{H}^1,$ and $g\in \dot{H}^{-1}\cap L^{\infty}$ with $\|g\|_{\dot{H}^{-1}}\leq 1/2,~\|g\|_{L^{\infty}}+\int_{\R^2} |g|\langle x \rangle^2\,dx\leq C_0$, for any $\epsilon\in(0, 1/2]$ we have
\begin{align}\label{patch_1}
\biggl|\int_{\R^2} gf^2dx\biggr|
&\leq C\bigl(\|g\|_{\dot{H}^{-1}}
\bigl|\ln\epsilon\bigr|^{1/2}+\epsilon\big)\big(\|f\|_{L^2(B_1)}+\|\nabla f\|_{L^2}\big)^2,\\
\biggl|\int_{\R^2} gf_1f_2dx\biggr|
&\leq C\bigl(\|g\|_{\dot{H}^{-1}}
\bigl|\ln\epsilon\bigr|^{1/2}+\epsilon\big)\big(\|f_1\|_{L^2(B_{1})}+\|\nabla f_1\|_{L^2}\big)\big(\|f_2\|_{L^2(B_1)}+\|\nabla f_2\|_{L^2}\big).\label{Eq.patch2}
\end{align}
\end{lem}

\begin{proof}
    Let $\eta(x)\in C_c^{\infty}(\R^2;[0,1])$ be fixed such that $\operatorname{supp}\eta\subset B_{1/4}$ and $\eta|_{B_{1/8}}=1$, and let $\eta_{\epsilon}(x)=\eta(\epsilon x)$ for $\varepsilon\in(0,1/2]$ and $x\in\R^2$. For $f\in \dot{H}^1$, we denote $$a_\varepsilon:=\left|B_{1/\varepsilon}\right|^{-1}\int_{B_{1/\varepsilon}}f(y)\,dy,\quad\forall\ \varepsilon\in(0,1/2].$$
    Then Poincar\'e's inequality implies that
    \begin{equation}\label{Eq.Poincare1}
        \|\eta_\varepsilon(j_\varepsilon*f-a_\varepsilon)\|_{L^2(\R^2)}\leq C\varepsilon^{-1}\|\nabla f\|_{L^2(\R^2)}
    \end{equation}
    for some absolute constant $C>0$ independent of $f$ and $\varepsilon$. We claim that
    \begin{align}
        |a_\varepsilon|\leq C(\|f\|_{L^2(B_1)}&+|\ln\varepsilon|^{1/2}\|\nabla f\|_{L^2(\R^2)})\leq C|\ln\varepsilon|^{1/2}(\|f\|_{L^2(B_1)}+\|\nabla f\|_{L^2(\R^2)}),\label{Eq.a_epsilon}\\
        &g\in L^1(\R^2)\text{ with }\int_{\R^2} g(x)\,dx=0,\label{Eq.g_integral}\\
        \int_{\R^2} &\left|g(1-\eta_\varepsilon^2)(j_{\epsilon}*f-a_\varepsilon)^2\right|\,dx\leq C\varepsilon\|\nabla f\|_{L^2(\R^2)}^2\label{Eq.g_faraway}
    \end{align}
    for some constant $C>0$ depending only on $C_0$. 

    Now we prove \eqref{patch_1} by admitting \eqref{Eq.a_epsilon}, \eqref{Eq.g_integral} and \eqref{Eq.g_faraway}, whose proof will be given as soon as we finish the proof of \eqref{patch_1}. We write $f^2=-(f-j_\varepsilon*f)^2+(j_\varepsilon*f)^2+2f(f-j_\varepsilon*f)$. First of all, by \eqref{Eq.f-conv_L^2} we have
    \begin{equation}\label{Eq.patch_1_1}
        \int_{\R^2}\left|g(f-j_\varepsilon*f)^2\right|\,dx\leq \|g\|_{L^\infty}{\|}f-j_\varepsilon*f\|_{L^2}^2\leq C\varepsilon^2\|\nabla f\|_{L^2}^2.
    \end{equation}
    
    For the term $(j_\varepsilon*f)^2$, we write
    \begin{align*}
        (j_\varepsilon*f)^2&=(j_\varepsilon*f-a_\varepsilon)^2+a_\varepsilon^2+2a_\varepsilon(j_\varepsilon*f-a_\varepsilon)\\
        &=\eta_\varepsilon^2(j_\varepsilon*f-a_\varepsilon)^2+(1-\eta_\varepsilon^2)(j_\varepsilon*f-a_\varepsilon)^2+a_\varepsilon^2\\
        &\qquad+2a_\varepsilon\eta_\varepsilon^2(j_\varepsilon*f-a_\varepsilon)+2a_\varepsilon(1-\eta_\varepsilon^2)(j_\varepsilon*f-a_\varepsilon).
    \end{align*}
    It follows from \eqref{Eq.mean_L^infty} and \eqref{Eq.Poincare1} that
    \begin{align*}
        &\left|\int_{\R^2}g\eta_\varepsilon^2(j_\varepsilon*f-a_\varepsilon)^2\,dx\right|\leq \|g\|_{\dot H^{-1}}\|\eta_\varepsilon^2(j_\varepsilon*f-a_\varepsilon)^2\|_{\dot H^1}\\
        \leq &C \|g\|_{\dot H^{-1}}\left(\|\nabla\eta_\varepsilon(j_\varepsilon*f-a_\varepsilon)\|_{L^\infty}\|\eta_\varepsilon(j_\varepsilon*f-a_\varepsilon)\|_{L^2}+\|\eta_\varepsilon^2(j_\varepsilon*f-a_\varepsilon)\|_{L^\infty}\|j_\varepsilon*\nabla f\|_{L^2}\right)\\
        \leq & C\|g\|_{\dot H^{-1}}\left(\varepsilon\|\nabla f\|_{L^2}|\ln\varepsilon|^{1/2}\cdot\varepsilon^{-1}\|\nabla f\|_{L^2}+\|\nabla f\|_{L^2}|\ln\varepsilon|^{1/2}\cdot\|\nabla f\|_{L^2}\right)\\
        \leq & C\|g\|_{\dot H^{-1}}|\ln\varepsilon|^{1/2}\|\nabla f\|_{L^2}^2.
    \end{align*}
    Also, by \eqref{Eq.g_integral} we have $\left|\int_{\R^2}ga_\varepsilon^2\,dx\right|=0$. Moreover,  \eqref{Eq.Poincare1} and \eqref{Eq.a_epsilon} imply that 
    \begin{align*}
        &\left|\int_{\R^2}ga_\varepsilon\eta_\varepsilon^2(j_\varepsilon*f-a_\varepsilon)\,dx\right|\leq \|g\|_{\dot H^{-1}}|a_\varepsilon|\|\eta_\varepsilon^2(j_\varepsilon*f-a_\varepsilon)\|_{\dot H^1}\\
        \leq & C\|g\|_{\dot H^{-1}}|a_\varepsilon|\left(\|\nabla\eta_\varepsilon\|_{L^\infty}\|\eta_\varepsilon(j_\varepsilon*f-a_\varepsilon)\|_{L^2}+\|j_\varepsilon*\nabla f\|_{L^2}\right)\\
        \leq & C\|g\|_{\dot H^{-1}}|a_\varepsilon|\left(\varepsilon\cdot \varepsilon^{-1}\|\nabla f\|_{L^2}+\|\nabla f\|_{L^2}\right)\leq C\|g\|_{\dot H^{-1}}|\ln\varepsilon|^{1/2}(\|f\|_{L^2(B_1)}+\|\nabla f\|_{L^2(\R^2)})^2.
    \end{align*}
    And by H\"older's inequality, \eqref{Eq.a_epsilon} and \eqref{Eq.g_faraway}, we have
    \begin{align*}
        &\int_{\R^2}\left|ga_\varepsilon(1-\eta_\varepsilon^2)(j_\varepsilon*f-a_\varepsilon)\right|\,dx\\
        \leq & |a_\varepsilon|\left(\int_{\R^2}|g|(1-\eta_\varepsilon^2)\,dx\right)^{1/2}\left(\int_{\R^2} \left|g(1-\eta_\varepsilon^2)(j_\varepsilon* f-a_\varepsilon)^2\right|\,dx\right)^{1/2}\\
        \leq & |a_\varepsilon|\left(\int_{|x|\geq 1/(8\varepsilon)}|g|\,dx\right)^{1/2}\sqrt\varepsilon\|\nabla f\|_{L^2}\leq C|a_\varepsilon|\left(\varepsilon^2\int_{\R^2}|g||x|^2\,dx\right)^{1/2}\sqrt\varepsilon\|\nabla f\|_{L^2}\\
        \leq & C|\ln\varepsilon|^{1/2}(\|f\|_{L^2(B_1)}+\|\nabla f\|_{L^2(\R^2)})\varepsilon\cdot \sqrt\varepsilon\|\nabla f\|_{L^2}
        \leq  C\varepsilon(\|f\|_{L^2(B_1)}+\|\nabla f\|_{L^2(\R^2)})^2.
    \end{align*}
    Therefore, we conclude that (also using \eqref{Eq.g_faraway})
    \begin{equation}\label{Eq.patch_1_2}
        \left|\int_{\R^2}g(j_\varepsilon*f)^2\,dx\right|\leq C\left(\|g\|_{\dot H^{-1}}|\ln\varepsilon|^{1/2}+\varepsilon\right)\left(\|f\|_{L^2(B_1)}+\|\nabla f\|_{L^2(\R^2)}\right)^2.
    \end{equation}

    As for the term $2f(f-j_\varepsilon*f)$, by H\"older's inequality, Lemma \ref{lem_fm} and \eqref{Eq.f-conv_L^2} we have 
    \begin{align}\label{Eq.patch_1_3}
        &\left|\int_{\R^2}gf(f-j_\varepsilon*f)\,dx\right|\leq\Bigl(\int_{\R^2}|g|^{3} \langle x \rangle^2\,dx\Bigr)^{\f13}\Bigl(\int_{\R^2} \frac{|f(x)|^6}{\langle x \rangle^4}\,dx\Bigr)^{\f16}\|f-j_\varepsilon*f\|_{L^2}\\
     \notag   \leq & CC_0\left(\|f\|_{L^2(B_1)}+\|\nabla f\|_{L^2(\R^2)}\right)\varepsilon\|\nabla f\|_{L^2(\R^2)}\leq C\varepsilon\left(\|f\|_{L^2(B_1)}+\|\nabla f\|_{L^2(\R^2)}\right)^2,
      \end{align} here we used that $\int_{\R^2}|g|^{3} \langle x \rangle^2\,dx\leq 
      \|g\|_{L^{\infty}}^2\int_{\R^2} |g|\langle x \rangle^2\,dx\leq C_0^3$. \if0\begin{align*}     
        &\Big(\|\nabla\eta_\varepsilon(j_\varepsilon*f-a_\varepsilon)\|_{L^\infty}\|\eta_\varepsilon(f-j_\varepsilon*f)\|_{L^2}+\|\eta_\varepsilon^2(\nabla j_\varepsilon)*f\|_{L^\infty}\|f-j_\varepsilon*f\|_{L^2}\\
        &\qquad\qquad\qquad\qquad+\|\eta_\varepsilon^2(j_\varepsilon*f-a_\varepsilon)\|_{L^\infty}\|\nabla f-j_\varepsilon*\nabla f\|_{L^2}\Big)\\
        \leq & C\|g\|_{\dot H^{-1}}\Big(\varepsilon\|\nabla f\|_{L^2}|\ln\varepsilon|^{1/2}\cdot\varepsilon\|\nabla f\|_{L^2}+\varepsilon^{-1}\|\nabla f\|_{L^2}\cdot \varepsilon\|\nabla f\|_{L^2}+\|\nabla f\|_{L^2}|\ln\varepsilon|^{1/2}\cdot\|\nabla f\|_{L^2}\Big)\\
        \leq & C\|g\|_{\dot H^{-1}}|\ln\varepsilon|^{1/2}\|\nabla f\|_{L^2}^2.
    \end{align*}
    It follows from \eqref{Eq.g_faraway} and \eqref{Eq.f-conv_L^2} that
    \begin{align*}
        &\int_{\R^2}\left|g(1-\eta_\varepsilon^2)(j_\varepsilon*f-a_\varepsilon)(f-j_\varepsilon*f)\right|\,dx\\
        \leq & \left(\int_{\R^2} \left|g(1-\eta_\varepsilon^2)(j_\varepsilon* f-a_\varepsilon)^2\right|\,dx\right)^{1/2}\left(\int_{\R^2} \left|g(1-\eta_\varepsilon^2)(f-j_\varepsilon* f)^2\right|\,dx\right)^{1/2}\\
        \leq & C\sqrt{\varepsilon}\|\nabla f\|_{L^2}\cdot \|g\|_{L^\infty}^{1/2}\|f-j_\varepsilon* f\|_{L^2}\leq C\varepsilon\|\nabla f\|_{L^2}^2.
    \end{align*}
    Also, by \eqref{Eq.a_epsilon} we have
    \begin{align*}
        \left|\int_{\R^2}ga_\varepsilon(f-j_\varepsilon*f)\,dx\right|&\leq |a_\varepsilon|\|g\|_{\dot H^{-1}}\|\nabla f-j_\varepsilon*\nabla f\|_{L^2}\\
        &\leq  C\|g\|_{\dot H^{-1}}|\ln\varepsilon|^{1/2}(\|f\|_{L^2(B_1)}+\|\nabla f\|_{L^2})\|\nabla f\|_{L^2}.
    \end{align*}
    Therefore, we conclude that
    \begin{equation}\label{Eq.patch_1_3}
        \left|\int_{\R^2}g(j_\varepsilon*f)(f-j_\varepsilon*f)\,dx\right|\leq C\left(\|g\|_{\dot H^{-1}}|\ln\varepsilon|^{1/2}+\varepsilon\right)\left(\|f\|_{L^2(B_1)}+\|\nabla f\|_{L^2(\R^2)}\right)^2.
    \end{equation}\fi
    Hence \eqref{patch_1} follows from \eqref{Eq.patch_1_1}, \eqref{Eq.patch_1_2} and \eqref{Eq.patch_1_3}. Similarly, one can prove \eqref{Eq.patch2}.
\end{proof}

To complete the proof of Lemma \ref{uni_H^-1_patch}, it suffices to prove \eqref{Eq.a_epsilon}, \eqref{Eq.g_integral} and \eqref{Eq.g_faraway}.
\begin{proof}[Proof of \eqref{Eq.a_epsilon}]
    It follows by taking $R=1/\varepsilon$ in \eqref{Eq.f_R-f_1} and using $|a_1|=|f_{(1)}|\leq C\|f\|_{L^2(B_1)}$.
\end{proof}

\begin{proof}[Proof of \eqref{Eq.g_integral}]
    By $\int_{\R^2} |g|\langle x \rangle^2\,dx<+\infty$ we have $g\in L^1(\R^2)$. Thus $ \hat{g}\in C(\R^2)$. Since $g\in\dot H^{-1}(\R^2)$, 
    we have $ |\xi|^{-1}\hat{g}(\xi)\in L^2(\R^2)$. If $\int_{\R^2}g\,dx\neq0$ then $ \hat{g}(0)\neq0$ and there exists $r>0$ such that $ |\hat{g}(\xi)|\geq |\hat{g}(0)|/2$ for $|\xi|\leq r$. Thus $\int_{B_r}||\xi|^{-1}\hat{g}(\xi)|^2d\xi\geq(|\hat{g}(0)|/2)^2\int_{B_r}|\xi|^{-2}d\xi=+\infty $, which contradicts $ |\xi|^{-1}\hat{g}(\xi)\in L^2(\R^2)$. Thus $ \hat{g}(0)=0$ and $\int_{\R^2}g\,dx=0$.
\end{proof}

\begin{proof}[Proof of \eqref{Eq.g_faraway}]
    We denote $\tilde f=f-a_\varepsilon$, then $j_\varepsilon*f-a_\varepsilon=j_\varepsilon*\tilde f$ {(as $\int_{\R^2}j(x)dx=1$)}. By  H\"older's inequality, we get
\begin{align*}
\int_{\R^2} |g(1-\eta^2_{\epsilon})(j_{\epsilon}*f-a_\varepsilon)^2|\,dx
&=\int_{\R^2} |(1-\eta_{\epsilon}^2)gx^{3/2}|\cdot\frac{|(j_{\epsilon}*\tilde f)(x)|^2}{|x|^{3/2}}\,dx\\
&\leq \Bigl(\int_{|x|\geq\frac{1}{8\epsilon}}|g|^{4/3}|x|^2\,dx\Bigr)^{\f34}\Bigl(\int_{|x|\geq\frac{1}{8\epsilon}} \frac{|(j_{\epsilon}*\tilde f)(x)|^8}{|x|^{6}}\,dx\Bigr)^{\f14}.
\end{align*}
It follows from  Lemma \ref{lem_fm} with a scaling argument and Poincar\'e's inequality that
\begin{align*}
&\Bigl(\int_{|x|\geq\frac{1}{\epsilon}} \frac{|(j_{\epsilon}*\tilde f)(x)|^8}{|x|^{6}}\,dx\Bigr)^{\f14}
\leq C\epsilon\bigl(\varepsilon\|j_{\epsilon}*\tilde f\|_{L^2(B_{1/\varepsilon})}+\|\nabla (j_{\epsilon}*\tilde f)\|_{L^2(\R^2)}\bigr)^2\\
\leq& C\varepsilon\bigl(\varepsilon\|\tilde f\|_{L^2(B_{1/\varepsilon})}+\varepsilon\|\tilde f-j_{\epsilon}*\tilde f\|_{L^2(\R^2)}+\|\nabla \tilde f\|_{L^2(\R^2)}\bigr)^2\leq C\varepsilon\|\nabla\tilde f\|_{L^2(\R^2)}^2=C\varepsilon\|\nabla f\|_{L^2(\R^2)}^2.
\end{align*}
Then by $\int_{\R^2}|g|^{4/3}|x|^2\,dx\leq 
      \|g\|_{L^{\infty}}^{1/3}\int_{\R^2} |g||x|^2\,dx\leq C_0^{4/3}$, we obtain \eqref{Eq.g_faraway}.
\if0\begin{equation}\label{exter}
\int_{\R^2} |g(1-\eta_{\epsilon}^2)(j_{\epsilon}*f)^2|\,dx
\leq C\epsilon\|g\|^{\f14}_{L^{\infty}}\Bigl(\int_{\R^2}|g||x|^2\,dx\Bigr)^{\f34}\bigl(\epsilon\|f\|_{L^2(B_{1/{\epsilon}})}+\|\nabla f\|_{L^2}\bigr)^2.
\end{equation}\fi 
\end{proof}

\subsection{Proof of weak-strong uniqueness}
\begin{proof}[Proof of Theorem \ref{thm_weak_unique}]
Here we follow the proof from Prange and Tan (\cite{PT}, Proposition 4.1). We focus on the case when $T\in(0,1/2)$ is small. We denote $\delta\!\rho:=\rho-\bar{\rho}$ and $\delta\!u:=u-\bar{u}$. Then the system for ($\delta\!\rho,\delta\!u$) reads 
{(here $\dot{\bar{u}}:=\partial_t\bar{u}+\bar{u}\cdot\nabla \bar{u}$)}
\begin{equation}\label{Eq.delta u}
     \left\{
     \begin{array}{l}
     \partial_t\delta\!\rho+\bar{u}\cdot\nabla \delta\!\rho+\delta\!u\cdot\nabla\rho=0,\qquad (t,x)\in \mathbb{R}^{+}\times\mathbb{R}^{2},\\
     \rho(\partial_t\delta\!u+u\cdot\nabla \delta\!u)-\Delta \delta\!u+\nabla \delta\!P=-\delta\!\rho\dot{\bar{u}}-\rho\delta\!u\cdot\nabla\bar{u},\\
     \dive \delta\!u = 0,\\
     (\delta\!\rho, \delta\!u)|_{t=0} =(0, 0).
     \end{array}
     \right.
\end{equation}

We set $\phi:=(-\Delta)^{-1}\delta\!\rho$ so that $\|\nabla\phi\|_{L^2(\R^2)}=\|\delta\!\rho\|_{\dot{H}^{-1}(\R^2)}$. Testing the first equation of \eqref{Eq.delta u} against $\phi$ and using the fact $\dive \bar{u}=\dive\delta\!u=0$, following the arguments in \cite{PT}, for all $t\in[0,T]$ we have
\begin{align}
&D(t)\leq \|\rho_0\|^{\f12}_{L^{\infty}}\|\sqrt{\rho}\delta\! u\|_{L^2((0,t)\times\R^2)}\exp\Big({\|\nabla \bar{u}\|_{L^1(0,T;L^{\infty}(\R^2))}}\Big)\leq C\|\sqrt{\rho}\delta\! u\|_{L^2((0,t)\times\R^2)},\label{Dt2}
\end{align}
where
\begin{align}
&D(t):=\sup_{0<s\leq t}s^{-\f12}\|\delta\!\rho(s,\cdot)\|_{\dot{H}^{-1}(\R^2)}.\label{Dt1}
\end{align}

Next we use a duality argument to control the difference $\sqrt{\rho}\delta\! u$ in $L^2((0,t)\times\R^2)$. Let $v(t,x)$ be the solution of the following linear backward parabolic system
\begin{equation}\label{Eq.v_backward}
     \left\{
     \begin{array}{l}
     \rho(\partial_tv+u\cdot\nabla v)+\Delta v+\nabla Q=\rho\delta\!u,\\
     \dive u = 0,\\
     v|_{t=T} =0.
     \end{array}
     \right.
\end{equation}
Then there holds 
\begin{equation}\begin{aligned}\label{v_estimate}
&\sup_{t\in[0,T]}\|(\sqrt{\rho}v,\nabla v)(t,\cdot)\|^2_{L^2(\R^2)}+\int_0^T\Bigl(\|\nabla v(t,\cdot)\|^2_{L^2}+\|(\sqrt{\rho}\dot{v},\nabla^2 v,\nabla Q)(t,\cdot)\|^2_{L^2}\Bigr)\,dt\\
&\leq C_0\|\sqrt{\rho}\delta\! u\|^2_{L^2((0,T)\times\R^2)}.   
\end{aligned}\end{equation}
Testing the equation of $\delta\!u$ in the system \eqref{Eq.delta u} by $v$ yields that
\begin{align}\label{estimate_delta_u}
\|\sqrt{\rho}\delta\! u\|^2_{L^2((0,T)\times\R^2)}\leq \int_0^T\bigl|\langle\delta\!\rho\dot{\bar{u}},v\rangle\bigr|\,dt+\int_0^T\bigl|\langle\rho\delta\!u\cdot\nabla\bar{u},v\rangle\bigr|\,dt. 
\end{align}
Firstly, we have (by Lemma \ref{lem4} and \eqref{v_estimate})
\begin{align*}
&\int_0^T\bigl|\langle\rho\delta\!u\cdot\nabla\bar{u},v\rangle\bigr|\,dt\leq\int_0^T\|\sqrt{\rho}\delta\!u\|_{L^2} \|\sqrt{\rho}v\|_{L^4} \|\nabla\bar{u}\|_{L^{4}}\,dt,\\
&\leq \int_0^T\|\sqrt{\rho}\delta\!u\|_{L^2} \|(\sqrt{\rho}v,\nabla v)\|_{L^2} \|\nabla\bar{u}\|^{\f12}_{L^{2}}\|\nabla^2\bar{u}\|^{\f12}_{L^{2}}\,dt,\\
&\leq T^{\f14}\|\sqrt{\rho}\delta\! u\|_{L^2((0,T)\times\R^2)}\|(\sqrt{\rho}v,\nabla v)\|_{L^{\infty}(0,T;L^2)}\|\nabla\bar{u}\|^{\f12}_{L^{\infty}(0,T;L^2)}\|\nabla^2\bar{u}\|^{\f12}_{L^2((0,T)\times\R^2)}\\
&\leq CT^{\f14}\|\sqrt{\rho}\delta\! u\|^2_{L^2((0,T)\times\R^2)},
\end{align*}
where $C>0$ depends only on $\|\rho_0\|_{L^{\infty}},\|\sqrt{\rho_0}u_0\|_{L^2}$ and $\|\nabla u_0\|_{L^2}$. Hence, by adjusting $T\in(0, 1/2)$ to smaller if necessary, it follows {from} \eqref{estimate_delta_u} that
\begin{equation}\label{estimate_delta_u_1}
    \|\sqrt{\rho}\delta\! u\|^2_{L^2((0,T)\times\R^2)}\leq 2\int_0^T\bigl|\langle\delta\!\rho\dot{\bar{u}},v\rangle\bigr|\,dt.
\end{equation}

{Using \eqref{Dt2}}, by adjusting $T\in(0, 1/2)$ to smaller if necessary,
 we have $D(t)\in(0,1/2)$ for $t\in (0,T]$.
Taking $\varepsilon=\| \delta\!\rho(t)\|_{\dot{H}^{-1}}$ in Lemma \ref{H-1_uni_bubble} if $\rho_0$ satisfies (H2) (or in Lemma \ref{uni_H^-1_patch} if $\rho_0$ satisfies (H1)), {and using Lemma \ref{lem4}} {and \eqref{Dt1}}, we obtain
\begin{align}
\int_0^T\bigl|\langle\delta\!\rho\dot{\bar{u}},v\rangle\bigr|\,dt &\leq C\int_0^T  t^{-\f12}\|\delta\!\rho(t)\|_{\dot{H}^{-1}}|\ln \varepsilon|^{\f12} \|\bigl(t^{\f12}\sqrt{\bar{\rho}}\dot{\bar{u}},t^{\f12}\nabla\dot{\bar{u}}\bigr)\|_{L^2} \|(\sqrt{\rho}v,\nabla v)\|_{L^2} \,dt\nonumber\\
&\leq C \|(\sqrt{\rho}v,\nabla v)\|_{L^{\infty}(0,T;L^2)}
\int_0^T D(t)|\ln ({\sqrt{t}}D(t))|^{\f12}\|\bigl(t^{\f12}\sqrt{\bar{\rho}}\dot{\bar{u}},t^{\f12}\nabla\dot{\bar{u}}\bigr)\|_{L^2}\,dt.\label{Eq.unique_est}
\end{align}
Thus, it follows from \eqref{Dt2}, \eqref{v_estimate}, \eqref{estimate_delta_u_1} {and \eqref{Eq.unique_est}} that
\begin{align*}
D(t)\leq C \int_0^t D(s)|\ln ({\sqrt{s}}D(s))|^{\f12}\bigl(\|s^{\f12}\sqrt{\bar{\rho}}\dot{\bar{u}}\|_{L^2}+s^{\f12}\|\nabla\dot{\bar{u}}\|_{L^2}\bigr)\,ds.
\end{align*}
Let $\gamma(t):= \sqrt t\bigl(\|\sqrt{\bar{\rho}}\dot{\bar{u}}\|_{L^2}+\|\nabla\dot{\bar{u}}\|_{L^2}\bigr)$, then by Lemma \ref{dot u} we have
\begin{align*}
\sup_{t\in[0,T]} t\|\sqrt{\bar{\rho}}\dot{\bar{u}}\|^2_{L^2}+\int_0^T t\|\nabla \dot{\bar{u}}\|^2_{L^2} dt \leq C\Longrightarrow   \int_0^{T} t\bigl(\|\sqrt{\bar{\rho}}\dot{\bar{u}}\|^2_{L^2}+\|\nabla\dot{\bar{u}}\|^2_{L^2}\bigr) dt \leq C,
\end{align*}   
hence $\gamma\in L^2([0,T])$, $ |\ln t|^{\f12}\gamma(t)\in L^1([0,T])$, $T,D(t)\in(0, 1/2)$ and we also have
\begin{align*}
D(t)\leq C\int_0^t D(s)|\ln ({\sqrt{s}}D(s))|^{\f12}\gamma(s) ds\leq C\int_0^t D(s)|\ln D(s)|^{\f12}|\ln s|^{\f12}\gamma(s) ds,\quad\forall\ t\in[0, T]. 
\end{align*}
Furthermore, the function $r\mapsto r|\ln r|^{\f12}$ is increasing near $0^+$ and satisfies $\int_0^{1/2} r^{-1}|\ln r|^{-1/2}\,dr=+\infty$. Hence we can apply Osgood's lemma (\cite{BCD}, Lemma 3.4) to conclude that $D(t)=0$ on $[0,T]$. By \eqref{estimate_delta_u_1}, \eqref{Eq.unique_est} {and \eqref{Dt1}}, we have
\begin{align*}
 \sqrt{\rho}\delta\!u\equiv0\quad\andf {\delta\!\rho}\equiv0 \text{~on~} [0,T].   
\end{align*}
Now, a direct $L^2$ estimate of \eqref{Eq.delta u} gives $\int_0^T\|\nabla \delta\!u(s,\cdot)\|^2_{L^2(\R^2)}\,ds=0$, and by  $\|\delta\!u\|_{L^2(B_r)}\leq C(\|\sqrt{\rho}\delta\!u\|_{L^2}+\|\nabla \delta\!u\|_{L^2})=0$ for all $r>0$ (see Lemma \ref{lem4}), we know that $u=\bar{u},~\rho=\bar{\rho}$ a.e. in $[0,T]\times \R^2$. 

The uniqueness on the whole time interval $[0,\infty)$ can be obtained by a bootstrap argument.
\end{proof}

\appendix 
 
 
 
 

\section{Technical lemmas}\label{Appen_proof}

The following lemma comes from Theorem B.1 in \cite{PL}.

\begin{lem}\label{lem_fm}
For $m\in [2,\infty)$ and $\ell\in (1+\frac{m}{2},\infty)$, there exists a positive constant $C$ depending only on $m$ and $\ell$ such that for all $f\in H^1(\R^2)$,
\begin{align}\label{fm}
\left(\int_{\R^2}\frac{|f|^m}{{\langle x \rangle}^2} 
(\log\langle x \rangle)^{-\ell} dx\right)^{1/m}
\leq C\big(\|f\|_{L^2(B_1)}+\|\nabla f\|_{L^2(\R^2)}\big).    
\end{align}
\end{lem}

The following technical lemma is {similar to Proposition A.4 in \cite{PT}}, which will be used constantly.  

\begin{lem}
\label{lem_Lp}
Assume that $\varrho\in L^{\infty}(\R^2)$ and $\int_{B(x_0,r_0)}\varrho(x)\,dx\geq c_0>0$. Then for all $p\in [1,\infty)$, there exists a constant $C>0$ depending only on $p$, $\|\varrho\|_{L^{\infty}}$ and $c_0$ such that for $f\in H^1_{\operatorname{loc}}$, 
\begin{align*}\label{f2}
\|f\|_{L^p(B(x_0,r_0))}\leq 
Cr_0^{2/p}\bigl((1+r_0)\|\nabla f\|_{L^2(B(x_0,r_0))}+\|\sqrt{\varrho}f\|_{L^2(B(x_0,r_0))}\bigr).    
\end{align*}
\end{lem}
\begin{proof}
Denote $\overline{f}:=\frac{1}{|B(x_0,r_0)|}\int_{B(x_0,r_0)}f(y)\,dy$. Then Poincar\'{e}'s inequality implies 
\begin{equation}\label{f1}
\|f-\overline{f}\|_{L^p(B(x_0,r_0))}\leq Cr_0^{2/p}\|\nabla f\|_{L^2(B(x_0,r_0))}.
\end{equation}
By \eqref{f1} for $p=2$, we get
\begin{align*}
\sqrt{c_0}|\overline{f}|&\leq \|\sqrt{\varrho}\overline{f}\|_{L^2(B(x_0,r_0))}
\leq \|\sqrt{\varrho}(f-\overline{f})\|_{L^2(B(x_0,r_0))}+\|\sqrt{\varrho}f\|_{L^2(B(x_0,r_0))}\\
&\leq C\|\varrho\|_{L^{\infty}}^{1/2}\|f-\overline{f}\|_{L^2(B(x_0,r_0))}+\|\sqrt{\varrho}f\|_{L^2(B(x_0,r_0))}\\
&\leq Cr_0\|\varrho\|_{L^{\infty}}^{1/2}\|\nabla f\|_{L^2(B(x_0,r_0))}+\|\sqrt{\varrho}f\|_{L^2(B(x_0,r_0))}.
\end{align*}
Therefore, we have
\begin{align*}
&\|f\|_{L^p(B(x_0,r_0))}\leq\|f-\overline{f}\|_{L^p(B(x_0,r_0))}+Cr_0^{2/p}|\overline{f}|\\ &\leq 
Cr_0^{2/p}\|\nabla f\|_{L^2(B(x_0,r_0))}+Cr_0^{2/p+1}\|\nabla f\|_{L^2(B(x_0,r_0))}+Cr_0^{2/p}\|\sqrt{\varrho}f\|_{L^2(B(x_0,r_0))}.
\end{align*}

This completes the proof.
\end{proof}

\begin{lem}\label{lem:app1}
Let $j_\epsilon$ be as in Subsection 4.1. 
There exists a constant $C>0$ such that for all $R>0$ and all $f\in \dot{H}^1(\R^2)$, there holds
\begin{align}\label{Eq.Appen_B1}
|(j_{\epsilon}*f)(x)-f_{(R)}|\leq C\|\nabla f\|_{L^2(\R^2)}\left|\ln\f{R}{\epsilon}\right|^{\f12},\quad \forall~ |x|\leq\f{R}{4},~ \forall~\epsilon\in\left(0,\f{R}{2}\right]
\end{align}
where $f_{(R)}:=\frac{1}{|B_R|}\int_{B_R} f(y)\,dy$. In particular, for all $\varepsilon\in(0,1/2]$, all $\varphi\in C_c^\infty(\R^2)$ satisfying $\operatorname{supp}\varphi\subset B_{1/(4\varepsilon)}$, we have
\begin{align}\label{Eq.mean_L^infty}
\|\varphi (j_{\epsilon}*f-f_{(1/\varepsilon)})\|_{L^{\infty}}\leq C|\ln{\epsilon}|^{\f12}\|\varphi\|_{L^{\infty}}\|\nabla f\|_{L^2(\R^2)}.  
\end{align}    
\end{lem}

\begin{proof}

Let $\zeta\in C_c^{\infty}(\R^2;[0,1])$ be fixed such that $\operatorname{supp}\zeta\subset B_1$ and $\zeta|_{B_{3/4}}=1$. Since $\int_{\R^2}j(x)dx=1$, we have $(j_{\epsilon}*f)(x)-f_{(1)}=\int_{\R^2}j_{\epsilon}(x-y)(f(y)-f_{(1)})\,dy=j_{\epsilon}*(f-f_{(1)})(x)$ for $x\in\R^2$. And by the definition of $\zeta$, for $|x|\leq1/4$ we have
\begin{align*}
(j_{\epsilon}*f)(x)-f_{(1)}=\zeta(x)\Bigl((j_{\epsilon}*f)(x)-f_{(1)}\Bigr)= j_{\epsilon}*\Bigl(\zeta\bigl(f-f_{(1)}\bigr)\Bigr)(x):=j_{\epsilon}*(\zeta g)(x),  
\end{align*}
where $g=f-f_{(1)}\in \dot{H}^1$ satisfies $\int_{B_1}g(y)\,dy=0$. By $\|j_\varepsilon*h\|_{L^\infty}\leq C|\ln \varepsilon|^{1/2}\|h\|_{H^1}$ and Poincar\'e inequality, we know that
\begin{align*}
\|j_{\epsilon}*(\zeta g)\|_{L^{\infty}}\leq C|\ln{\epsilon}|^{\f12} \|\zeta g\|_{H^1}\leq C|\ln{\epsilon}|^{\f12}(\|g\|_{L^2(B_1)}+\|\nabla g\|_{L^2(\R^2)})\leq C|\ln{\epsilon}|^{\f12}\|\nabla g\|_{L^2(\R^2)}.   
\end{align*}
This proves \eqref{Eq.Appen_B1} for $R=1$, and then a scaling argument gives \eqref{Eq.Appen_B1} for all $R>0$.
\end{proof}

\begin{lem}
    There exists a contant $C>0$ such that for all $f\in \dot H^1(\R^2)$, there holds
    \begin{equation}\label{Eq.f_R-f_1}
        \left|f_{(R)}-f_{(1)}\right|\leq C\|\nabla f\|_{L^2(\R^2)}|\ln R|^{1/2},\quad\forall\ R\geq2.
    \end{equation}
\end{lem}

\begin{proof}
    Fix 
    $\varepsilon_0=1/4$. Then Lemma \ref{lem:app1} implies that
    \begin{align*}
        \left|(j_{\varepsilon_0}*f)({0})-f_{(R)}\right|\leq C\|\nabla f\|_{L^2}\left|\ln \frac{R}{\varepsilon_0}\right|^{1/2}\leq C\|\nabla f\|_{L^2}|\ln R|^{1/2},\\
        \left|(j_{\varepsilon_0}*f)({0})-f_{(1)}\right|\leq C\|\nabla f\|_{L^2}\left|\ln {\varepsilon_0}\right|^{1/2}\leq C\|\nabla f\|_{L^2}|\ln R|^{1/2}
    \end{align*}
    for all $R\geq 2$. Then \eqref{Eq.f_R-f_1} follows from the triangle inequality.
\end{proof}

\medskip

\section* {Acknowledgments.}

D. Wei is partially supported by the National Key R\&D Program of China under the grant 2021YFA1001500. Z. Zhang is partially supported by NSF of China under Grant 12288101.

\end{document}